\documentclass{jaca}  
\usepackage{graphicx,subcaption,color,wasysym,hyperref}
\usepackage{enumitem} 
\makeatletter
\DeclareFontFamily{U}{tipa}{}
\DeclareFontShape{U}{tipa}{m}{n}{<->tipa10}{}
\newcommand{\arc@char}{{\usefont{U}{tipa}{m}{n}\symbol{62}}}%

\newcommand{\arc}[1]{\mathpalette\arc@arc{#1}}

\newcommand{\arc@arc}[2]{%
  \sbox0{$\m@th#1#2$}%
  \vbox{
    \hbox{\resizebox{\wd0}{\height}{\arc@char}}
    \nointerlineskip
    \box0
  }%
}
\makeatother

\begin{document}


\title{Optimal partitions for the sum and the maximum of eigenvalues}

\author{Beniamin Bogosel and Virginie Bonnaillie-No\"el}

\begin{abstract}
In this paper we compare the candidates to be spectral minimal partitions for two criteria: the maximum and the average of the first eigenvalue on each subdomains of the partition. We analyze in detail the square, the disk and the equilateral triangle. Using numerical simulations, we propose candidates for the max, prove that most of them can not be optimal for the sum and  then exhibit better candidates for the sum.
\end{abstract}

\KeysAndCodes{Minimal partitions, shape optimization, Dirichlet-Laplacian eigenvalues, curved polygons, numerical simulations}{49Q10, 35J05, 65K10, 65N06, 65N25}

\section{Introduction}
A great interest was shown lately towards problems concerning optimal partitions related to some spectral quantities (see \cite{BucButHen98,BucBut05,CafLin07,HelHofTer09}). Among them, we distinguish two problems which interest us. 
Let $\Omega$ be a bounded and connected domain and $\mathfrak P_{k}(\Omega)$ the set of partitions of $\Omega$ in $k$ disjoint and open subdomains $D_{i}$. We look for partitions $\mathcal D=(D_i)_{1\leq i\leq k}$ of $\Omega$ which minimize
\begin{description}
\item[{\bf\quad Problem 1.}] the largest first eigenvalue of the Dirichlet Laplace operator in $D_{i}$:
\begin{equation}
\mathfrak L_{k}(\Omega)= \min\left\{ \max_{1\leq i\leq k} \lambda_1(D_i),(D_{i})_{1\leq i\leq k}\in\mathfrak P_{k}(\Omega)\right\} . 
\end{equation}
\item[{\bf\quad Problem 2.}] the sum of the first eigenvalues of the Dirichlet Laplace operator in $D_{i}$:
\begin{equation}
\mathfrak L_{k,1}(\Omega)= \min\left\{\frac 1k\sum_{1\leq i\leq k}\lambda_1(D_i),(D_{i})_{1\leq i\leq k}\in\mathfrak P_{k}(\Omega)\right\}. 
\end{equation}
\end{description}

For simplicity we refer to these two problems in the sequel as minimizing the sum or the max. Theoretical results concerning the existence and regularity of the optimal partitions regarding these problems can be found in \cite{BucButHen98,CafLin07,HelHofTer09} (and references therein). Despite the increasing interest in these problems there are just a few cases where optimal partitions are known explicitly, and exclusively in the case of the max. 

\paragraph{Structure of the paper}
In the following section, we apply known results to obtain estimates of the energy of optimal partitions in the case of a disk, a square or an equilateral triangle and recall informations about the structure of an optimal partitions. 
Then, we focus on the minimization problem for the max: we give explicit results for $k=2,4$ and in other cases, we present the best candidates we obtained by using several numerical methods. 
In Section~\ref{sec.compar}, we recall and apply a criterion which shows when a partition optimal for the max cannot be optimal for the sum. 
In cases where the previous criterion shows that candidates for the max are not optimal for the sum, we propose better candidates in Section~\ref{sec.sum}. These candidates are either obtained with iterative methods already used in \cite{BouBucOud09,BBN16} or are constructed explicitly. 

\section{Applications of known results}
\subsection{Estimates of the energies}
By monotonocity of the $p$-norm, we easily compare the two optimal energies:
\begin{equation}
\frac 1k \mathfrak L_{k}(\Omega)\leq \mathfrak L_{k,1}(\Omega)\leq \mathfrak L_{k}(\Omega).
\end{equation}
More quantitative bounds can be obtained with the eigenmodes of the Dirichlet-Laplacian on $\Omega$. 
For $k\geq1$, we denote by $\lambda_{k}(\Omega)$ the $k$-th eigenvalue of the Dirichlet-Laplacian on $\Omega$ (arranged in increasing order and repeated with multiplicity) and by $L_k(\Omega)$ the smallest eigenvalue (if any) for which there exists an eigenfunction with $k$ nodal domains ({\it i.e.} the components of the nonzero set of the eigenfunction). In that case, the eigenfunction gives us a $k$-partition such that the first eigenvalue of the Dirichlet-Laplacian on each subdomain equals $L_{k}(\Omega)$.  We set $L_k(\Omega) =+\infty$ if there is no eigenfunction with $k$ nodal domains. 
It is standard to prove (see \cite{HelHofTer09} for example):
\begin{align}
\lambda_{k}(\Omega) & \leq {\mathfrak L}_{k}(\Omega) \leq L_{k}(\Omega),\label{ineq.HHOT}\\
\frac 1k\sum_{i=1}^k\lambda_{i}(\Omega)& \leq {\mathfrak L}_{k,1}(\Omega) \leq L_{k}(\Omega).\label{ineq.HHOTsum}
\end{align}
Let us consider $k=1,2$. Since the first eigenvalue of the Dirichlet-Laplacian is simple and $\Omega$ is connected, then necessarily any eigenfunction associated with $\lambda_{k}(\Omega)$ has one or two nodal sets whether $k=1$ or $2$. Consequently $L_{k}(\Omega)=\lambda_{k}(\Omega)$ for $k=1,2$ and  
\begin{equation}\label{eq.k12}
\mathfrak L_{k}(\Omega) = \lambda_{k}(\Omega)=L_{k}(\Omega), \quad \mbox{ when }k=1,2.
\end{equation}
Furthermore, any nodal partition associated with $\lambda_{k}(\Omega)$ is optimal for the max when $k=1,2$.
We say that an eigenfunction associated with $\lambda$ is {\it Courant-sharp} if it has $k$ nodal domains with $k=\min\{j,\lambda_{j}(\Omega)=\lambda\}$.\\

In this paper, we focus on three geometries: a square $\square$ of sidelength $1$, a disk $\ocircle$ of radius $1$ and an equilateral triangle $\triangle$ of sidelength $1$. The eigenvalues are explicit and given in Table~\ref{tab.vp} where  $j_{m,n}$ is the $n$-th positive zero of the Bessel function of the first kind $J_{m}$ and $\lambda_{k}(\Omega)$ is the $k$-th element of the set $\{\lambda_{m,n}(\Omega)\}$. In Table~\ref{tab.estim}, we explicit the lower and upper bounds in \eqref{ineq.HHOT}--\eqref{ineq.HHOTsum}. 

\begin{table}[h!]
\begin{center}\begin{tabular}{c | cl}
$\Omega$ & $\lambda_{m,n}(\Omega)$ & $m,n$\\
\hline
& \\[-8pt]
$\square$ & $\pi^2(m^2+n^2)$& $m,n\geq 1$\\[2pt]
$\triangle$ & $\frac{16}9 \pi^2(m^2+mn+n^2)$& $m,n\geq 1$\\[2pt]
$\ocircle$ & $j_{m,n}^2$ & $m\geq0$, $n\geq 1$ {(multiplicity 2 for $m\geq1$)}\\[2pt]
\end{tabular}\end{center}
\caption{Eigenvalues for the Dirichlet-Laplacian on $\Omega=\square,\ \triangle,\ \ocircle$. \label{tab.vp}}
\end{table}

\begin{table}[h!]
{\small\begin{center}\begin{tabular}{|c||r|r|r||r|r|r||r|r|r|}
\hline
& \multicolumn{3}{c||}{Square} & \multicolumn{3}{c||}{Disk}& \multicolumn{3}{c|}{Equilateral triangle}\\
\hline
 $k$ & \multicolumn{1}{c|}{$\frac1k\displaystyle\sum_{i=1}^k\lambda_{i}$} & \multicolumn{1}{c|}{$\lambda_{k}$} & \multicolumn{1}{c||}{$L_{k}$} & \multicolumn{1}{c|}{$\frac1k\displaystyle\sum_{i=1}^k\lambda_{i}$} & \multicolumn{1}{c|}{$ \lambda_{k}$} & \multicolumn{1}{c||}{$L_{k}$} & \multicolumn{1}{c|}{$\frac1k\displaystyle\sum_{i=1}^k\lambda_{i}$} & \multicolumn{1}{c|}{$\lambda_{k}$} & \multicolumn{1}{c|}{$L_{k}$} \\
 \hline
 1  &  19.74 &  19.74 & 19.74 & 5.78 &  5.78 & 5.78  &  52.64 &  52.64 & 52.64 \\
 2  &  34.54 &  49.35 & 49.35 & 10.23 & 14.68 & 14.68  &  87.73 & 122.82 & 122.82 \\
 3  &  39.48 &  49.35 & 98.70 & 11.72& 14.68& 74.89  &  99.43 & 122.82 & 228.10 \\
 4  &  49.35 &  78.96 & 78.96 & 15.38& 26.37 & 26.37  & 127.21 & 210.55 & 210.55 \\
 5  &  59.22 &  98.70 & 256.61 & 17.58 & 26.37 & 222.93 & 147.39 & 228.10 & 368.46 \\
 6  &  65.80 &  98.70 & 246.74 & 19.73 & 30.47 & 40.71  & 160.84 & 228.10 & 543.92 \\
 7  &  74.73 & 128.30 & 493.48 & 22.72 & 40.71 & 449.93   & 185.49 & 333.37 & 684.29 \\
 8  &  81.42 & 128.30 & 197.39 & 24.97 & 40.71 & 57.58  & 203.97 & 333.37 & 491.29 \\
 9  &  91.02 & 167.78 & 177.65 & 27.67& 49.22 & 755.89  & 222.25 & 368.46 & 473.74 \\
10  &  98.70 & 167.78 & 286.22 & 29.82 & 49.22 & 76.94  & 236.87 & 368.46 & 1000.12 \\
\hline
\end{tabular}\end{center}}
\caption{Bounds \eqref{ineq.HHOT}--\eqref{ineq.HHOTsum} when $\Omega=\square$, $\ocircle$ and $\triangle$.\label{tab.estim}}
\end{table}

\begin{rem}\label{rem.k4}
Looking at Table \ref{tab.sect}, we observe that the lower and upper bounds \eqref{ineq.HHOT} for the max are equal when $k=1,2,4$. Thus, in that case, the optimal energy for the max is $\lambda_{k}(\Omega)$ and any associated nodal partition is optimal. We can wonder if it can happen for larger $k$. We will come back to this question in the following section. 
\end{rem}

In the case of the disk, for an odd index $k$, the upper bound $L_k$ given in Table~\ref{tab.estim} corresponds to the $k$-th simple eigenvalue on the disk whose associated nodal partition has $k$ concentric annuli. This eigenvalue grows rapidly with $k$. In these cases, we may observe that considering partitions in $k$ equal sectors gives a better upper bound. If we denote by $\Sigma_{\alpha}$ the angular sector of opening $\alpha$ (see \cite{BonLen14} for analysis on angular sectors), this new upper bound writes
\[{\mathfrak L}_{k}(\Omega) \leq \lambda_{1}\left(\Sigma_{\frac{2\pi}k}\right)=j^2_{\frac k2,1}.\]
This heavily improved upper bound is estimated in Table~\ref{tab.sect}.
\begin{table}[h!]
\begin{center}
\begin{tabular}{|c|c|c|c|c|c|}
\hline
$k$ & 3 & 5 & 7 & 9\\
\hline
$\lambda_{1}\left(\Sigma_{\frac{2\pi}k}\right)$ & 20.19 & 33.22 & 48.83 & 66.95\\
\hline
\end{tabular}
\caption{Energy of the partitions with $k$ angular sectors.\label{tab.sect}}
\end{center}
\end{table}
We can notice that if we replace $L_{k}(\ocircle)$ in Table~\ref{tab.sect} by these values, we have monotonicity with respect to $k$ for the upper bound which is more consistent with the monotonicity of $\mathfrak L_{k}(\ocircle)$ and $\mathfrak L_{k,1}(\ocircle)$. \\

Bounds \eqref{ineq.HHOT}--\eqref{ineq.HHOTsum} provide a quantative information about the energy of the optimal partition but no information about the structure of the minimal partition. Let us know discuss this point.

\subsection{About the minimal partitions}
As we have mentioned previously, several works are dealing with the existence of optimal partitions (see \cite{BucButHen98,CafLin07,HelHofTer09}\ldots). These results also give information about their regularity.
\begin{thm} \label{thm.ex}
For any $k\geq1$, there exists a regular optimal $k$-partition for any of the two optimization problems. 
Furthermore any optimal partition for the max is an equipartition, {\it i.e.} the first eigenvalues on each subdomain are equal.
\end{thm}
Let us recall that a $k$-partition $\mathcal D$ is called \emph{regular} if its boundary, $N(\mathcal D)= \cup_{1\leq i\leq k}\partial D_{i}$, is locally a regular curve, except at a finite number of singular points, where a finite number of half-curves meet with equal angles. We say that $\mathcal D$ satisfies the \emph{equal angle meeting property}.\\

\section{Candidates for the max}
\label{sec.cmax}

\subsection{Explicit solutions}
The following result, established in \cite{HelHofTer09} gives information about the equality case in \eqref{ineq.HHOT} and permits to solve, in some cases, the optimization problem for the max. In that cases, the minimal partitions are nodal. 
\begin{thm}\label{thm.HHOT}
The nodal partition of a Courant-sharp eigenfunction is optimal for the max. Conversely, if a minimal partition for the max is nodal, the associated eigenfunction is Courant-sharp.
\end{thm}
This theorem means that if we have one equality in \eqref{ineq.HHOT}, then we have equality everywhere and any optimal partition is nodal.

As observed in \eqref{eq.k12}, we know that the optimal $k$-partition for the max is given by the $k$-th eigenfunction when $k=1,2$. Remark \ref{rem.k4} notices that it is still the case for $k=4$ for the three considered geometries. The following result established in \cite{HelHofTer09, MR3445517,BH2016LMP} show that for other $k$, the situation is very different.

\begin{prop}\label{prop.k24}
If $\Omega$ is a disk $\ocircle$, a square $\square$ or an equilateral triangle $\triangle$, then 
\[\lambda_{k}(\Omega)=\mathfrak L_{k}(\Omega)= L_k(\Omega) \qquad\mbox{if and only if}\qquad k=1,2,4.\]
Thus the minimal $k$-partition for the max is nodal if and only if $k=1,2,4$. 
\end{prop}
Figure~\ref{fig.PartNod} gives examples of optimal $k$-partitions for the max. Note that since $\lambda_{2}(\Omega)$ is double, the minimal $2$-partition is not unique whereas for $k=4$ we do have 	uniqueness.

\begin{figure}[h!]
\begin{center}
\begin{minipage}[b]{.6\linewidth}
\begin{center}
\includegraphics[height=1.4cm]{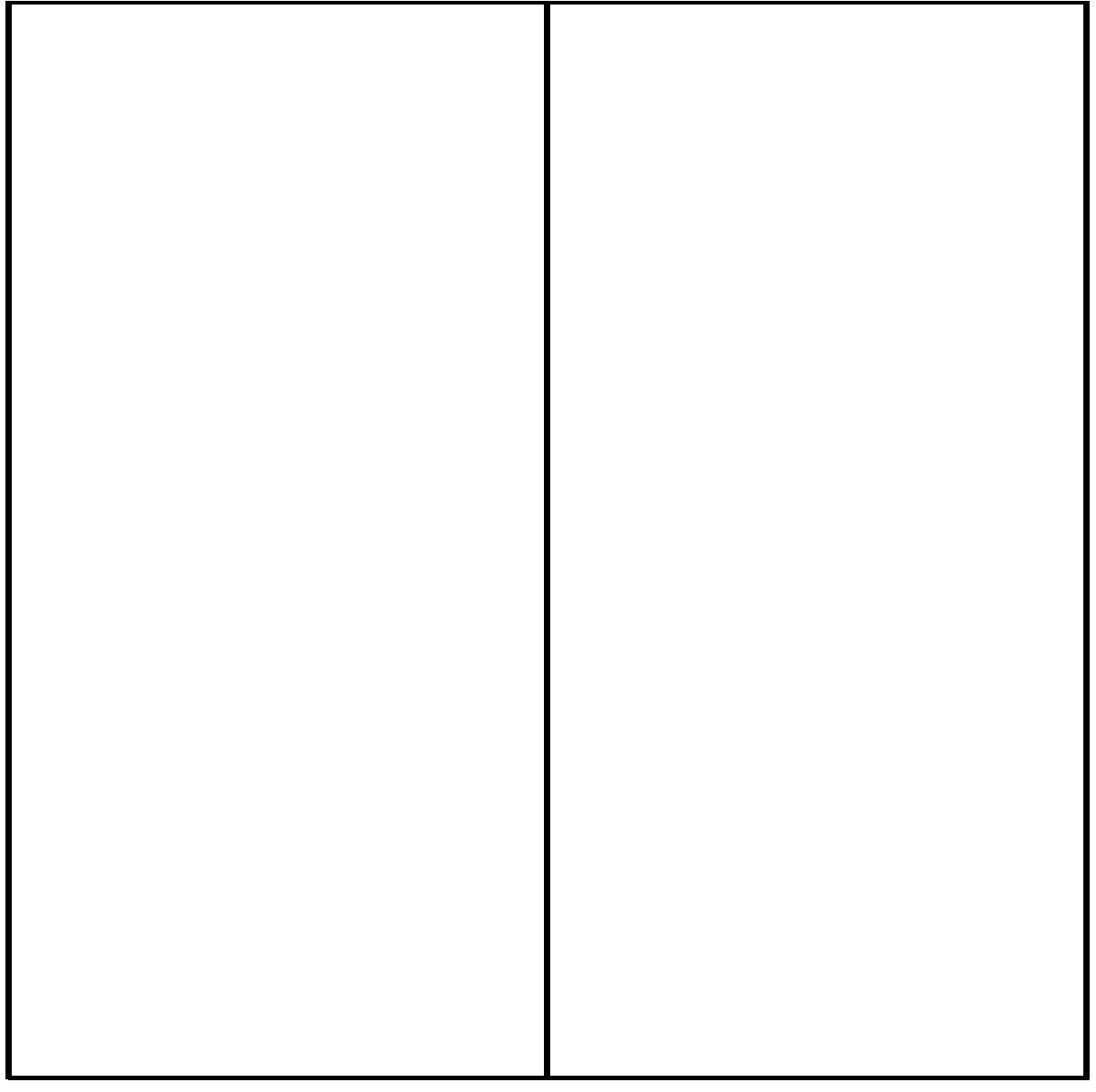}\
\includegraphics[height=1.4cm]{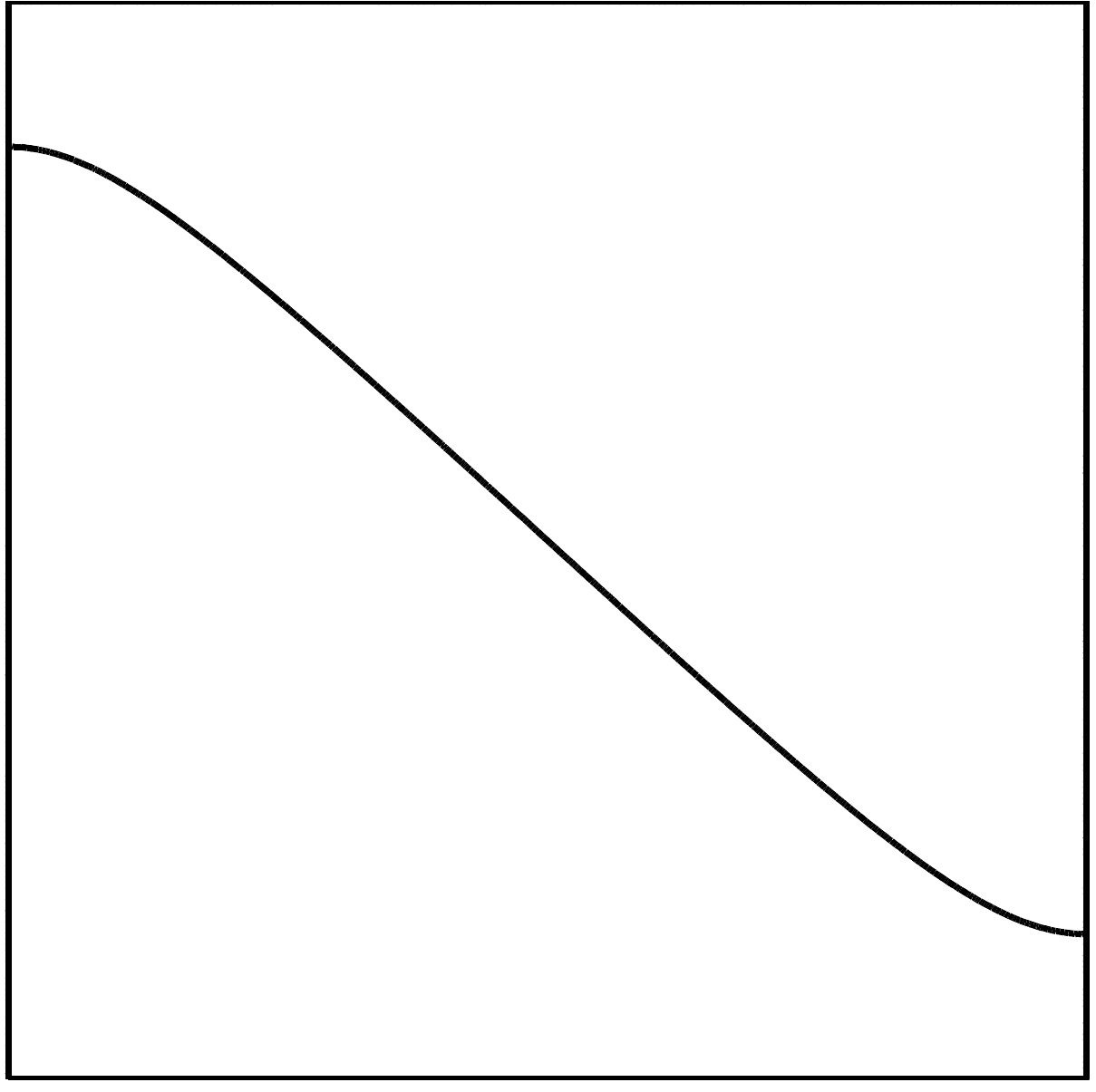}\ 
\includegraphics[height=1.4cm]{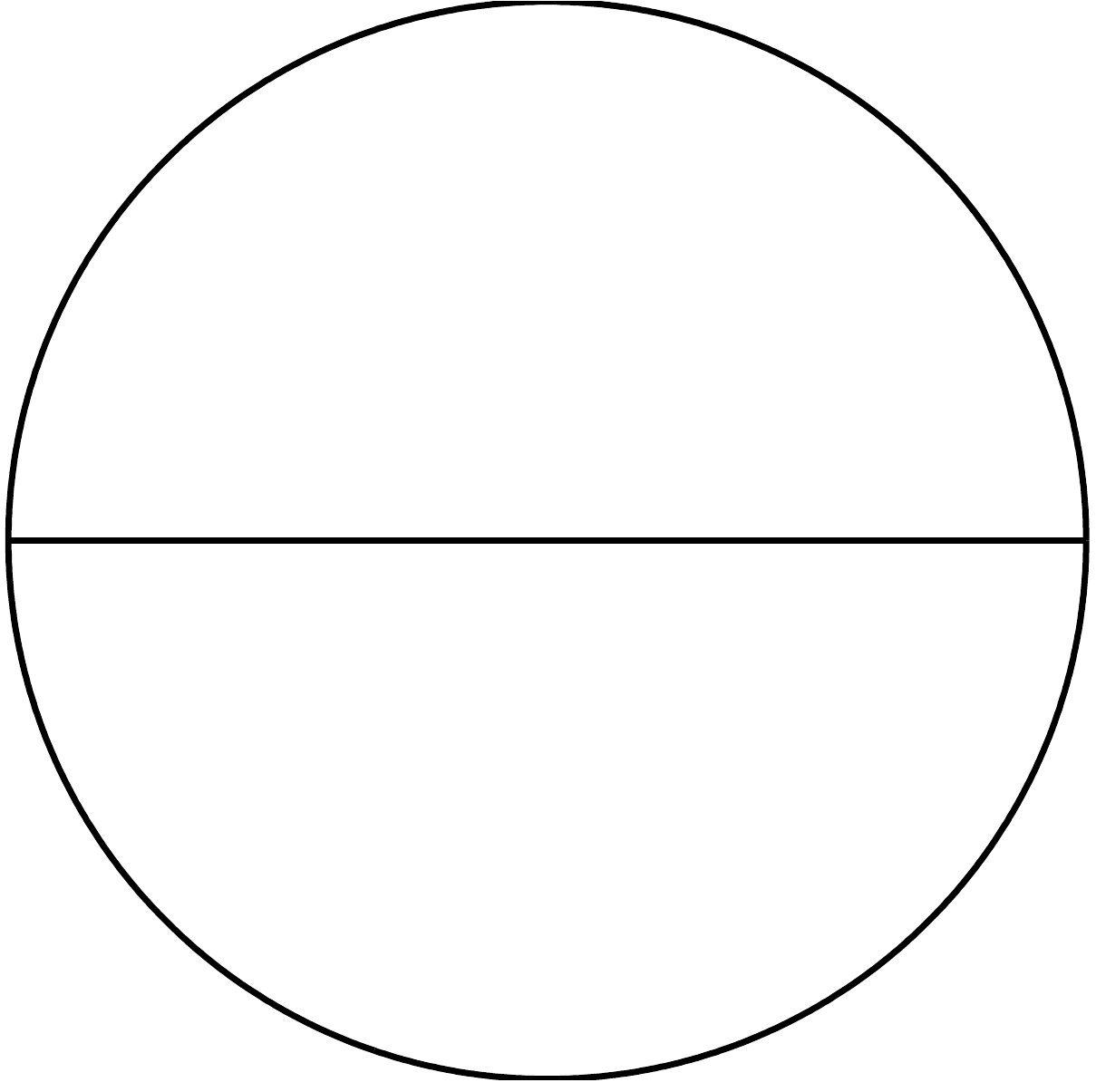}
\includegraphics[height=1.4cm]{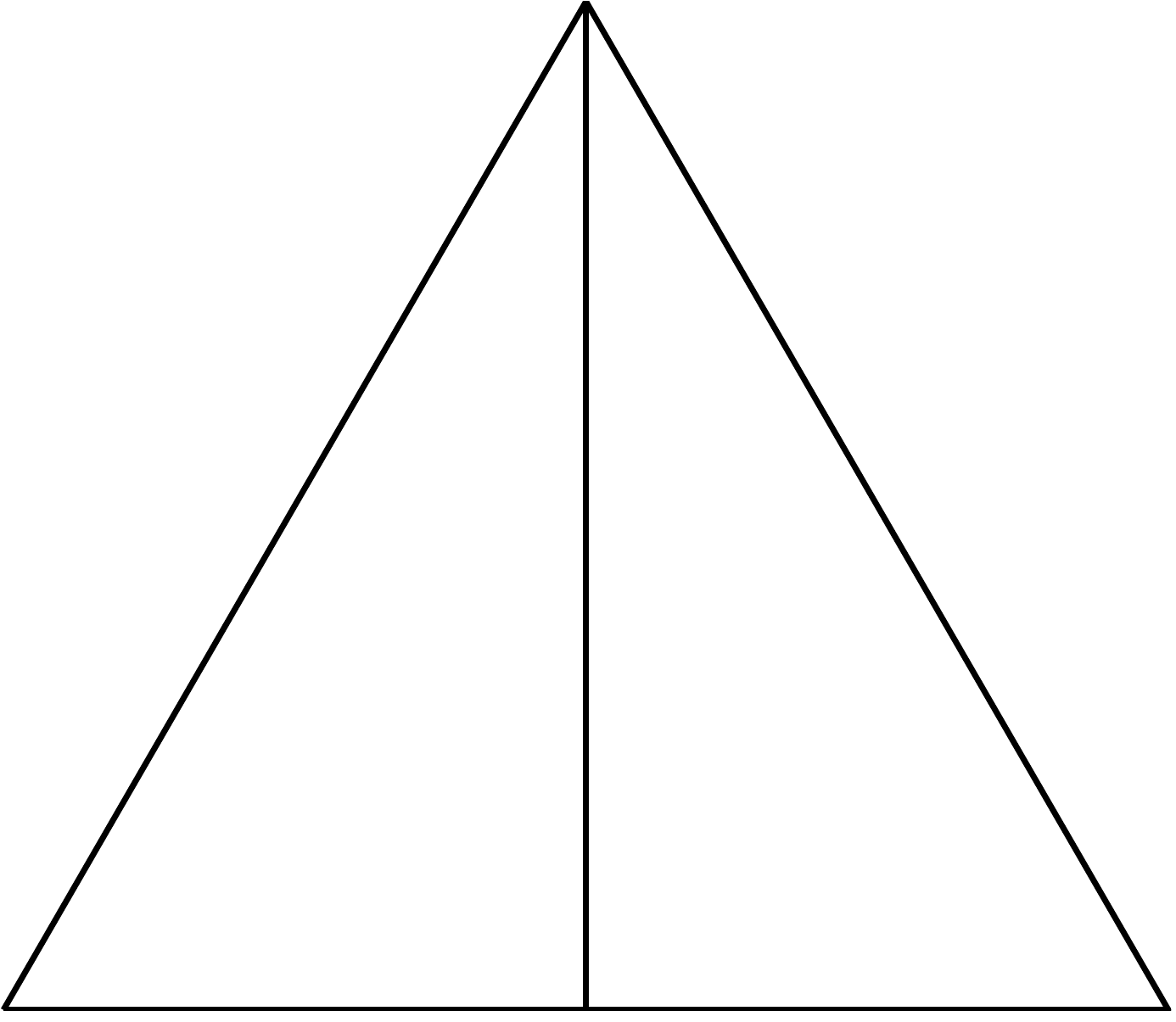}\
\includegraphics[height=1.4cm]{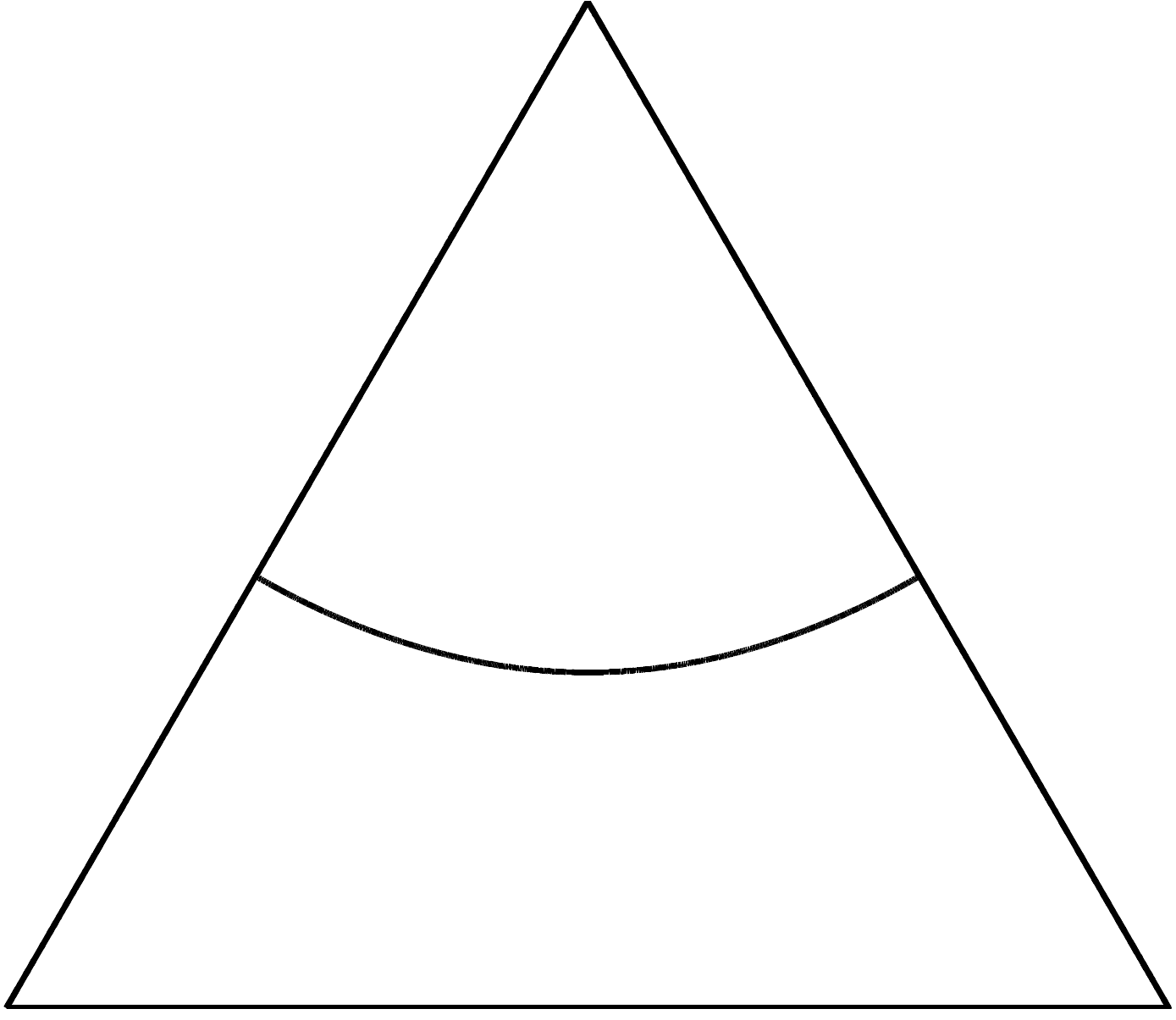}
\subcaption{$k=2$}
\end{center}
\end{minipage}
\hfill\begin{minipage}[b]{.38\linewidth}
\begin{center}
\includegraphics[height=1.4cm]{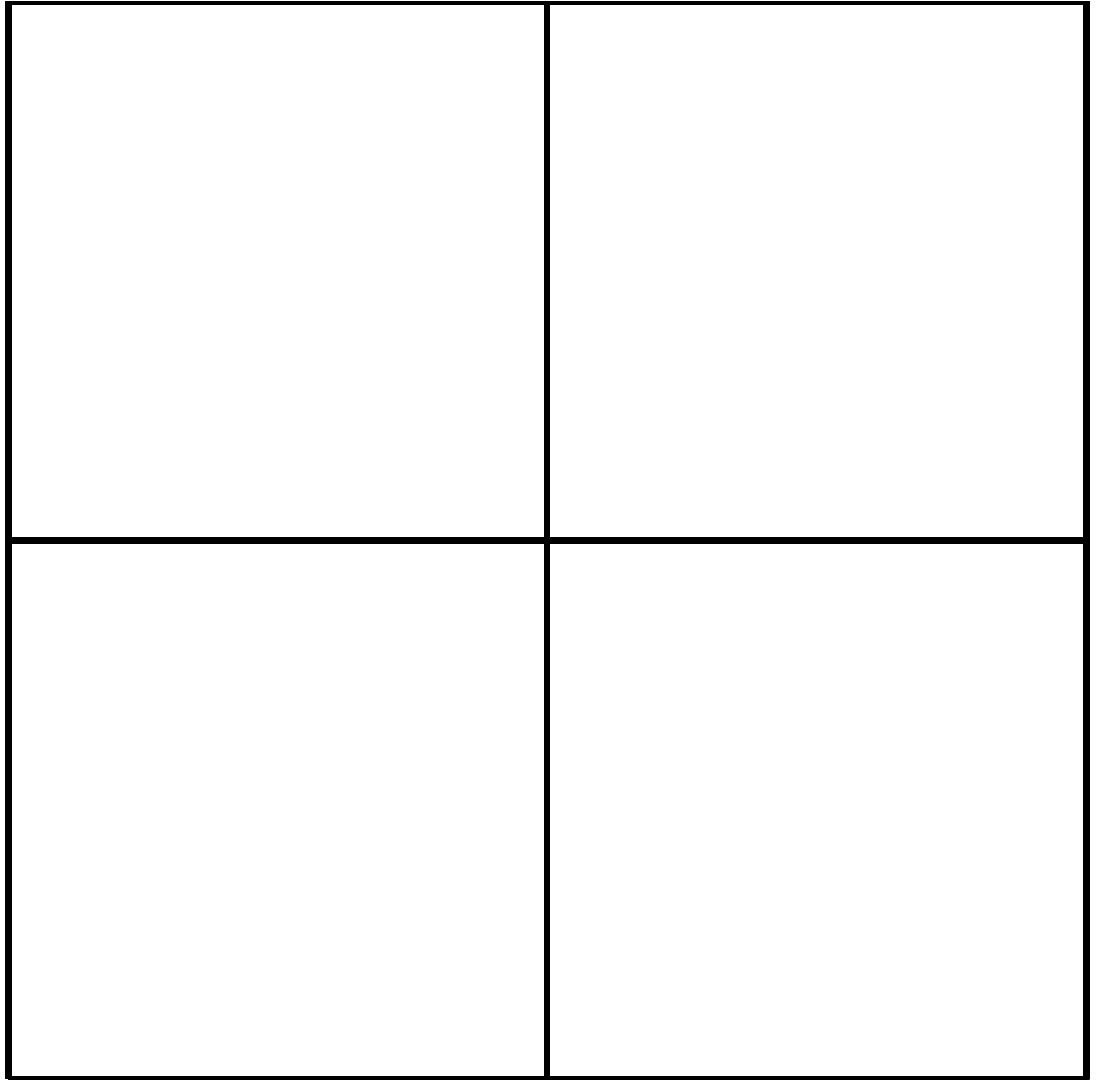}\
\includegraphics[height=1.4cm]{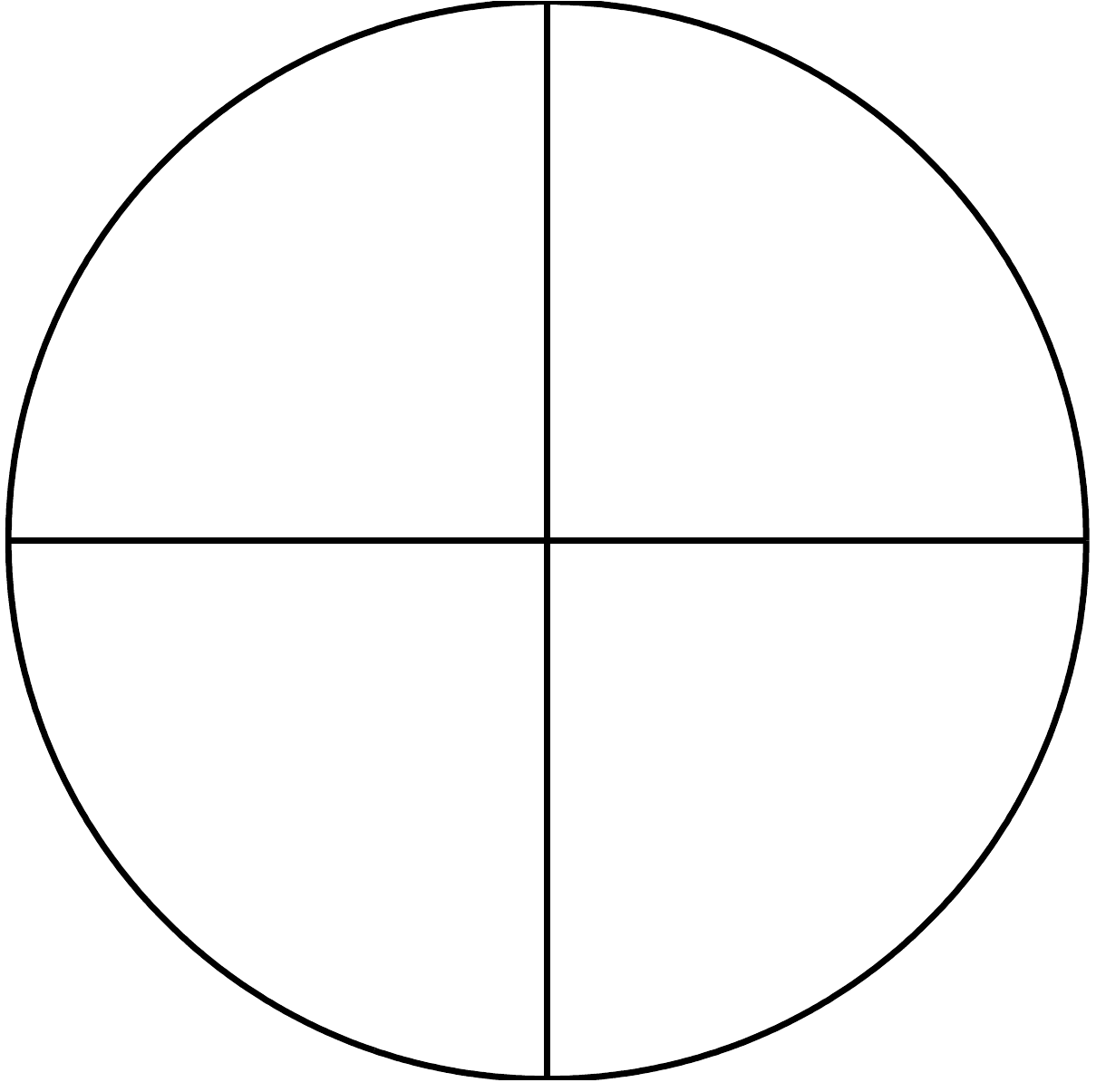}
\includegraphics[height=1.4cm]{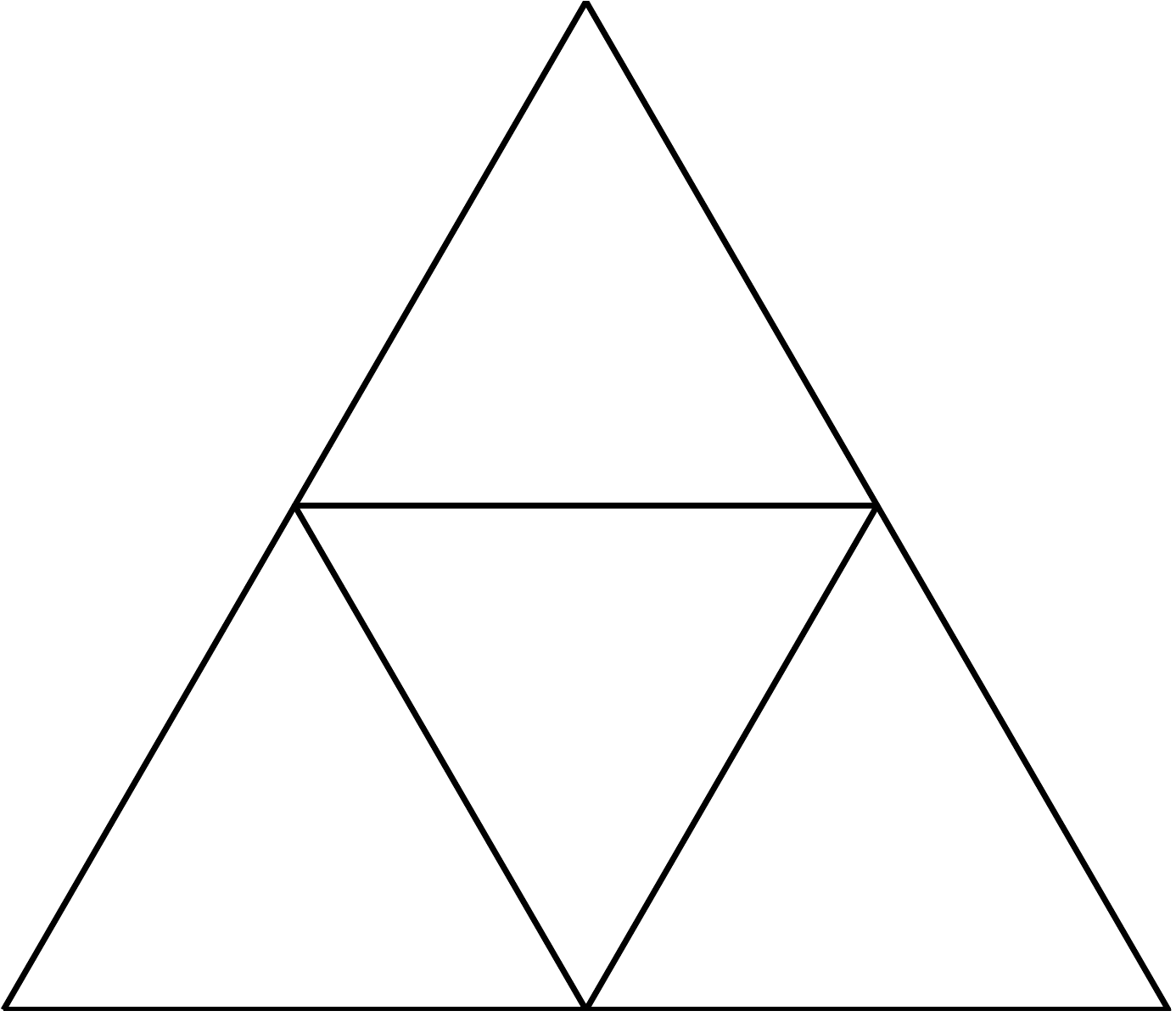}
\subcaption{$k=4$}
\end{center}
\end{minipage}
\end{center}
\caption{Optimal $k$-partitions for the max, $k=2,4$.\label{fig.PartNod}}
\end{figure}

\subsection{Candidates obtained by numerical simulations}

In cases when explicit solutions are not known, some numerical algorithms are available in order to compute minimal partitions for the max. There are two types of methods that have been used in the literature.
\begin{itemize}
\item The first option is to look for partitions which are nodal for a mixed Dirichlet-Neumann problem on $\Omega$ (see for example \cite{BHV}). In this case we are certain to have an equipartition. On the other hand, we force \emph{a priori} the structure of the partition.
\item Another option is to use an iterative approach based on an adaptation of the algorithm introduced in \cite{BouBucOud09}. In \cite{BBN16}, the authors presented an algorithm allowing the minimization of a $p$-norm of eigenvalues, which approaches the optimization problem for the max as $p$ is large. Another variant of the algorithm was presented, which penalizes the difference between eigenvalues on different cells of the partition. These iterative algorithms have the advantage that there is no constraint on the structure of the partition. On the other hand a relaxation method is used in order to compute the eigenvalues, which makes the method less precise.
\end{itemize}

In \cite{BBN16} the authors studied the square, the disk and the equilateral triangle. They used initially the iterative methods to detect the structure of the partition. Secondly, when some parts of the boundaries of the cells of the partitions were segments or easily parametrizable curves, the mixed Dirichlet-Neumann approach was used. Not surprisingly, the Dirichlet-Neumann method yields best results, {\it i.e.} the smallest maximal eigenvalue, when it can be used. We recall below the best partitions obtained with the two above methods when $k=3,5, 6, 7, 8, 9, 10$. 
The smallest energies obtained numerically are summed up in Table \ref{tab.Linfty} and the associated partitions are represented in Figures \ref{fig.triangleinf}, \ref{fig.squareinf} and \ref{fig.disqueinf}. In these figures, the partitions with green lines are obtained with the Dirichlet-Neumann approach whereas the others are the results of the iterative method with penalization.

\begin{table}[h!t]
\begin{center}
\begin{tabular}{|c|c|c|c|c|c|c|c|}
\hline
$k$ & 3 & 5 & 6 & 7 & 8 & 9 & 10\\
\hline
$\triangle$ &  $142.88$ & $251.99$ & $275.97$ & $345.91$ & $389.31$ & $428.75$ & $451.73$ \\
\hline
$\ocircle$ & $20.19$ & $33.21$ & $39.02$ & $44.03$ & $50.46$ & $58.25$ & $67.19$ \\
\hline
$\square$ & $66.58$ & $104.29$ & $127.11$ & $146.88$ & $161.28$ & $178.08$ & $204.54$ \\
\hline
\end{tabular}
\caption{Numerical estimates for $\mathfrak L_{k}(\Omega)$, $\Omega=\triangle, \square, \ocircle$.\label{tab.Linfty}}
\end{center}
\end{table}

\begin{figure}[h!]
\centering
\setlength{\tabcolsep}{1pt}
\begin{tabular}{ccccccc}
\includegraphics[width=0.138\textwidth]{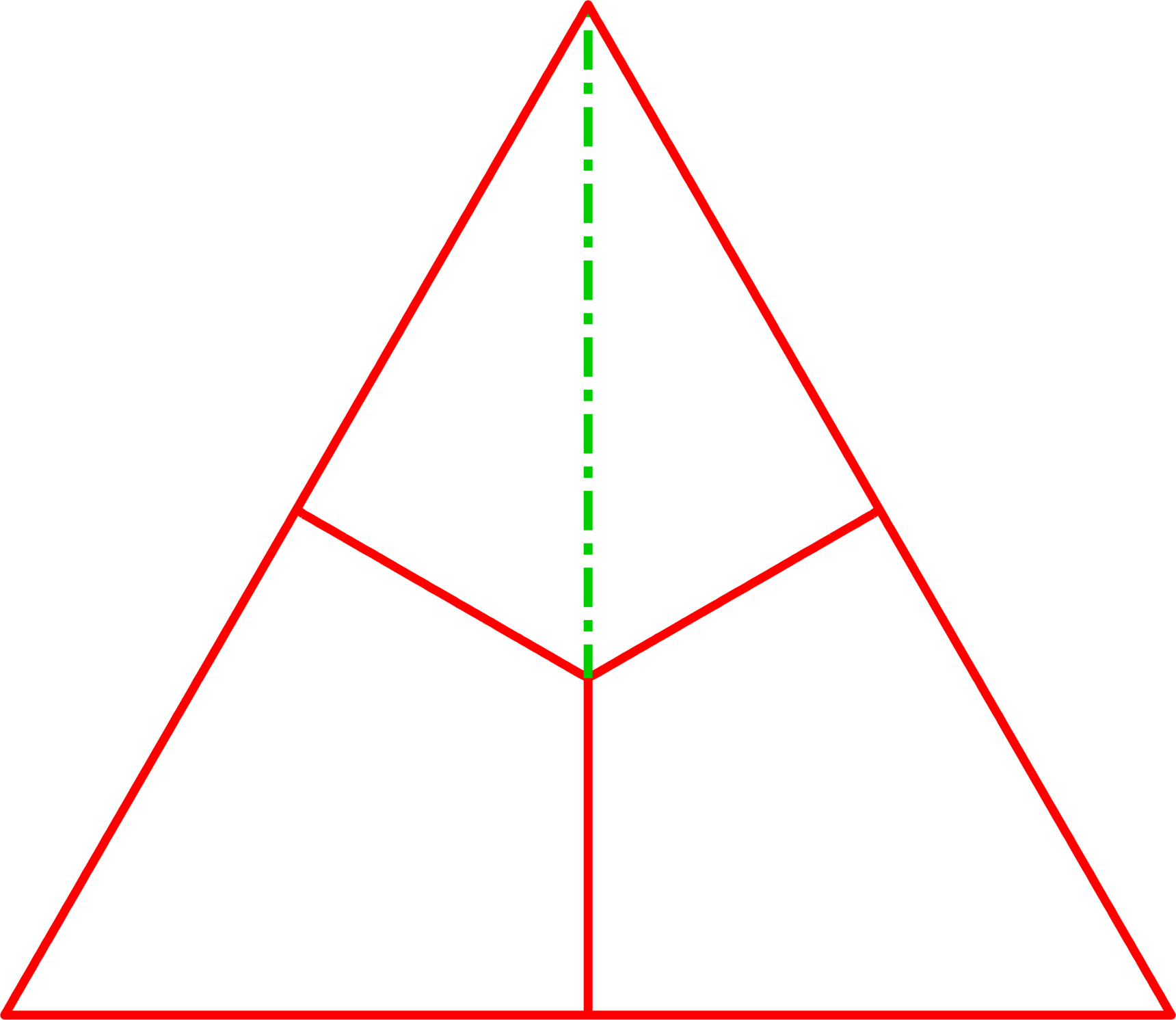} 
&\includegraphics[width =0.138\textwidth]{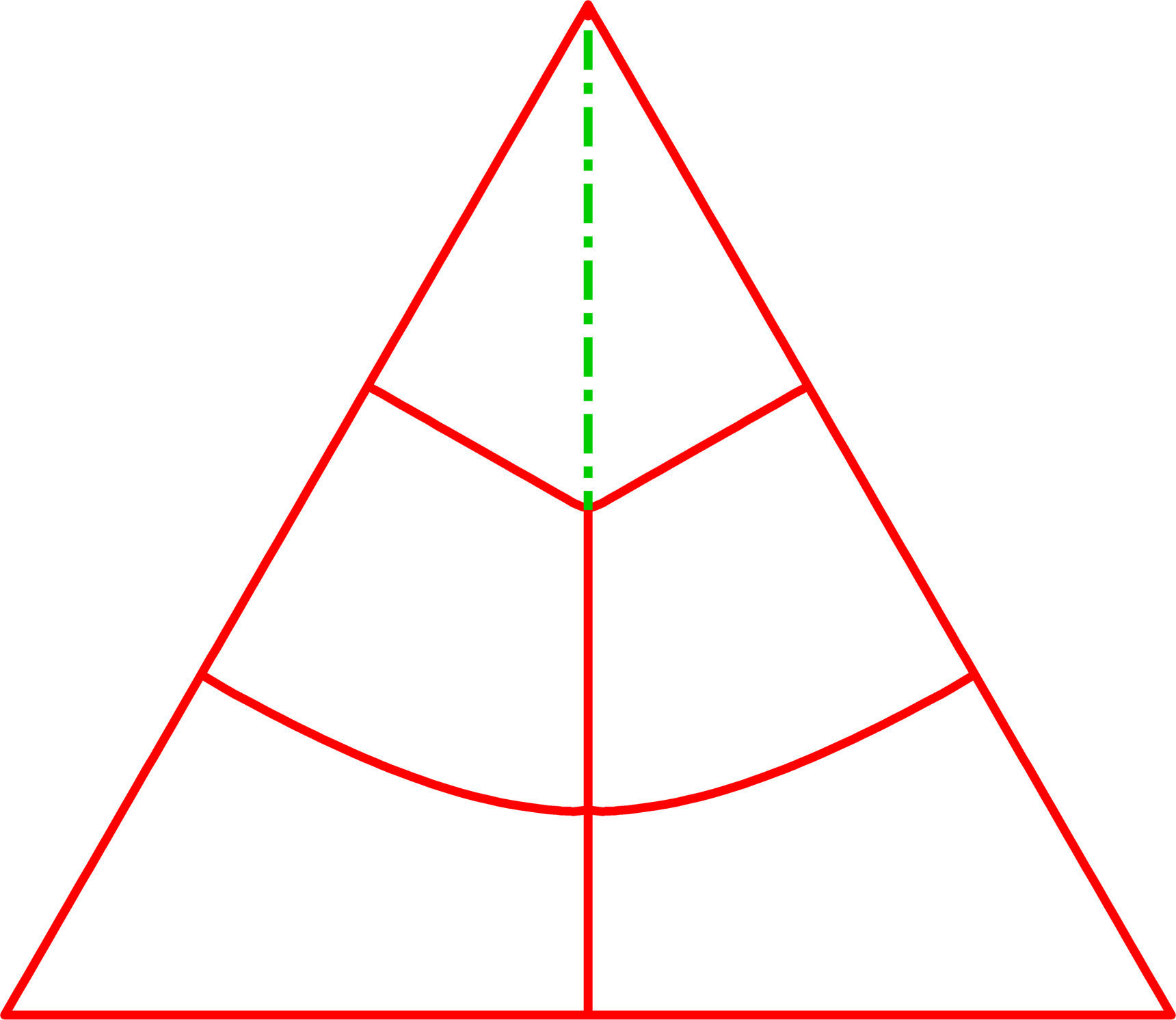}
&\includegraphics[width =0.138\textwidth]{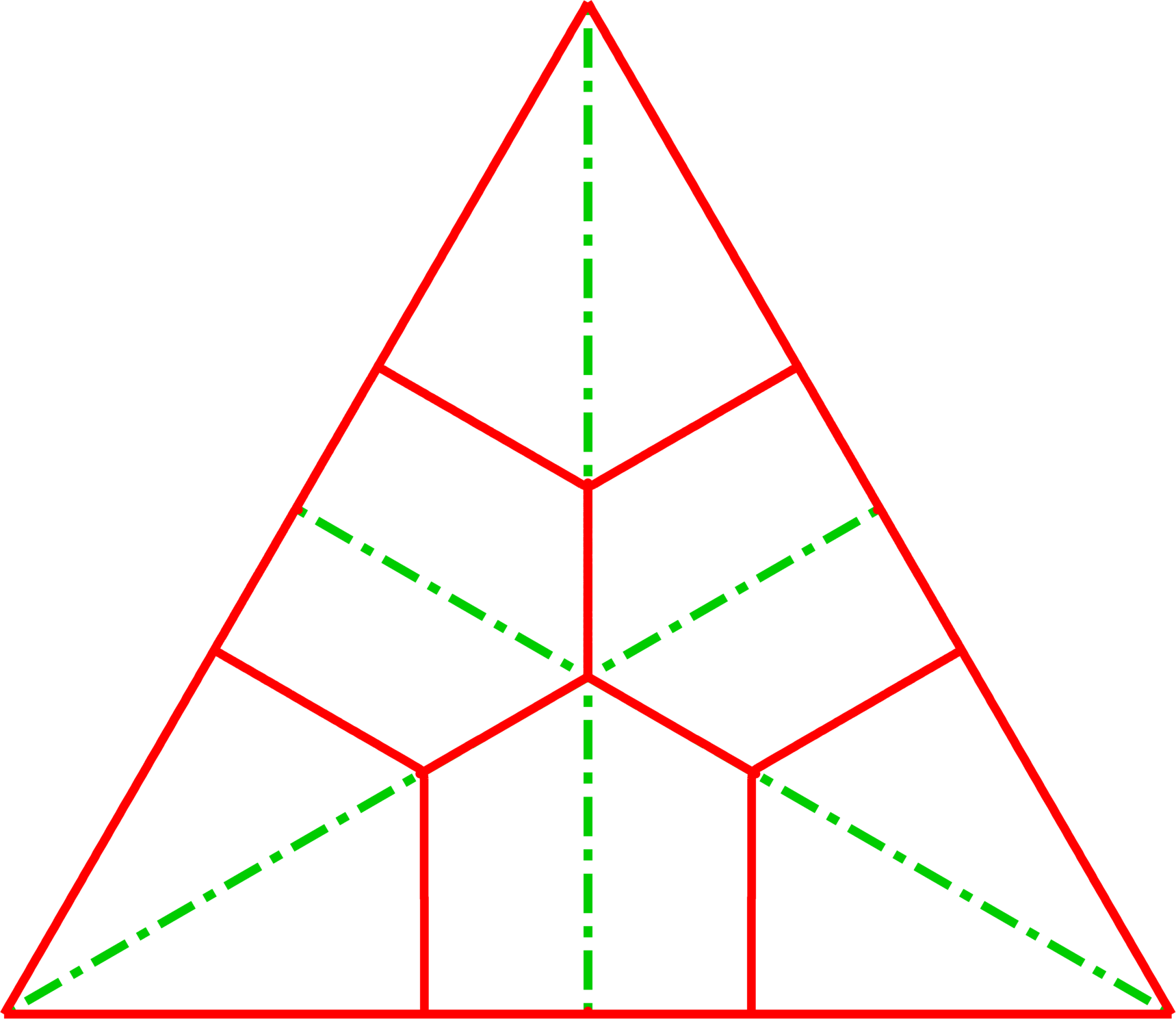}
&\includegraphics[width =0.138\textwidth]{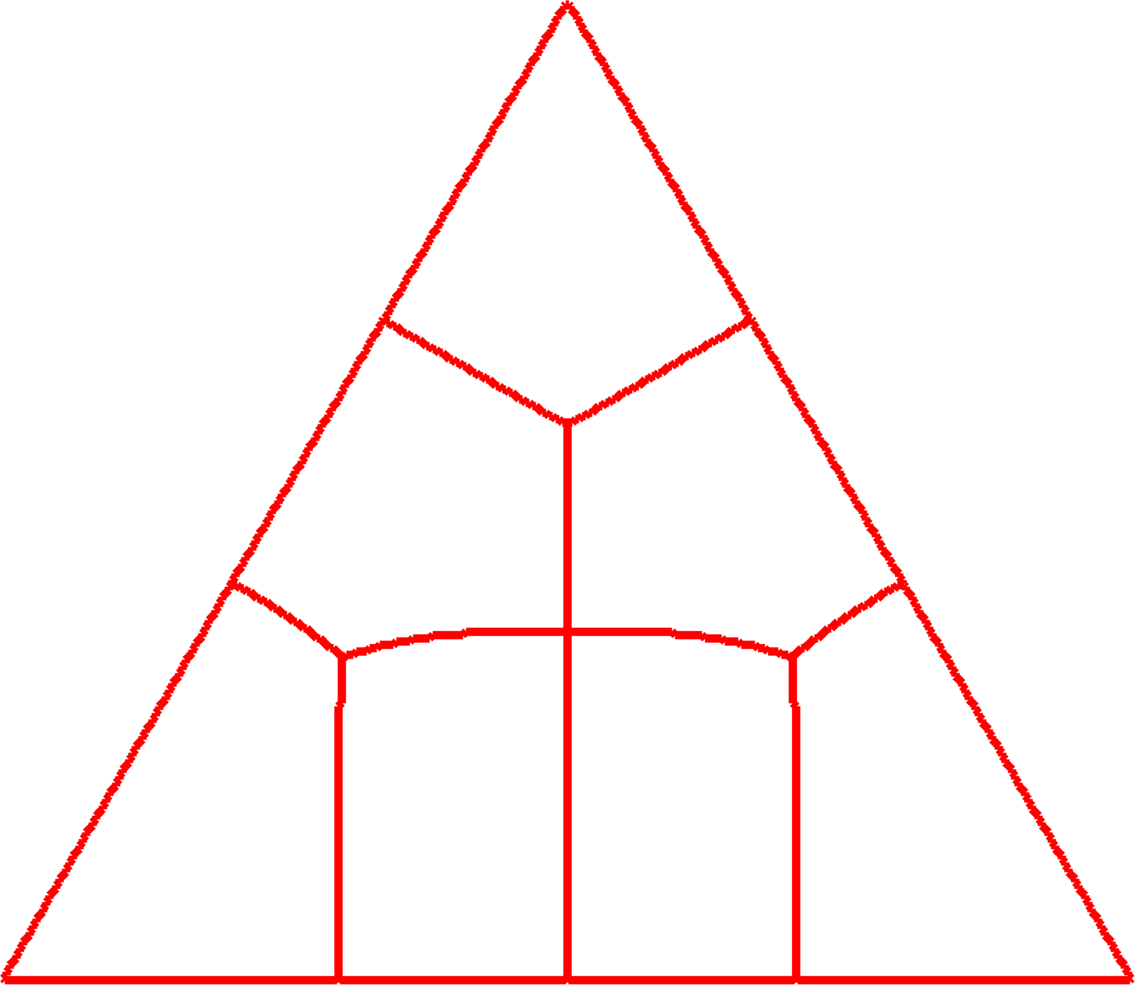}
&\includegraphics[width =0.138\textwidth]{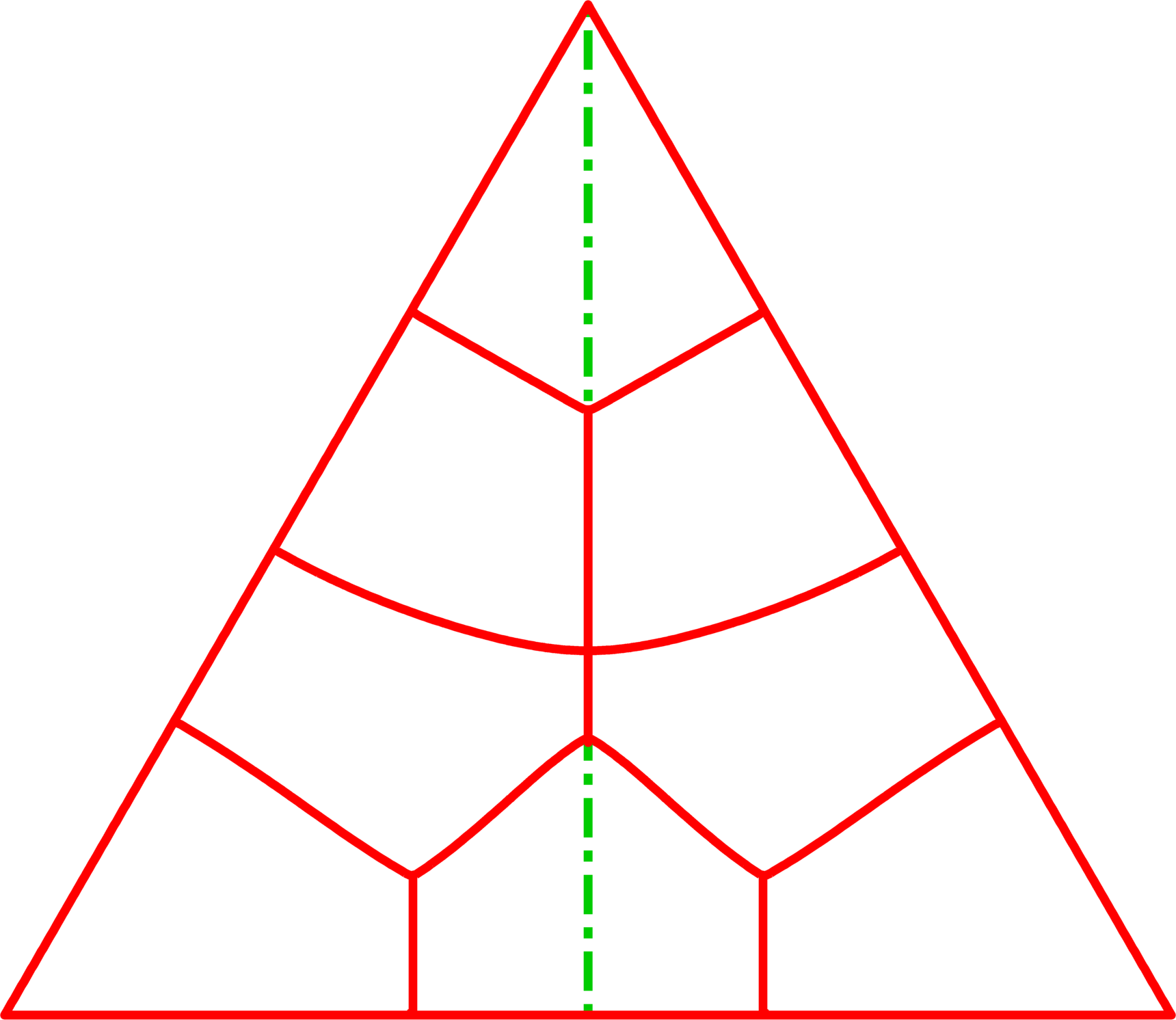}
&\includegraphics[width =0.138\textwidth]{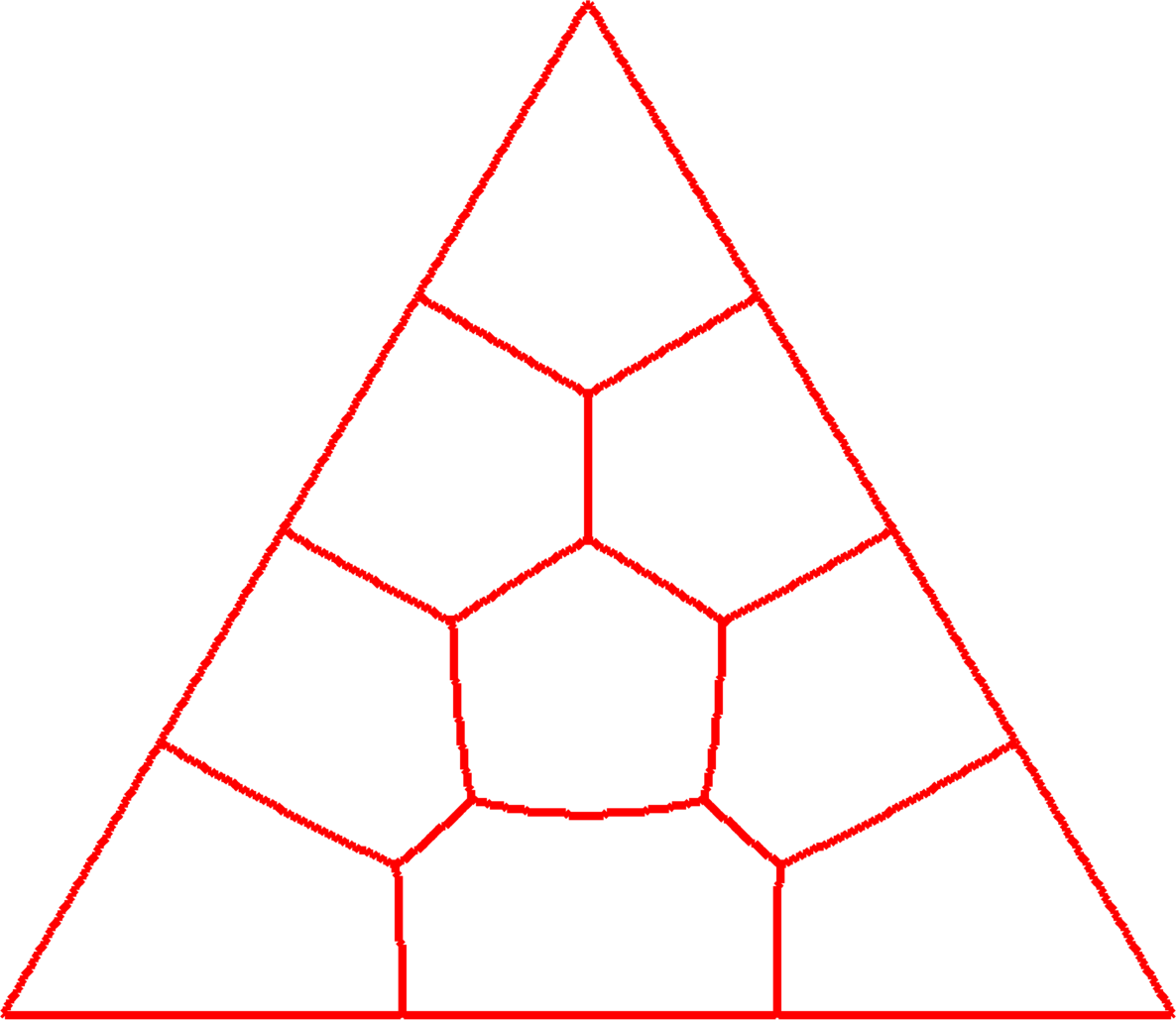}
&\includegraphics[width =0.138\textwidth]{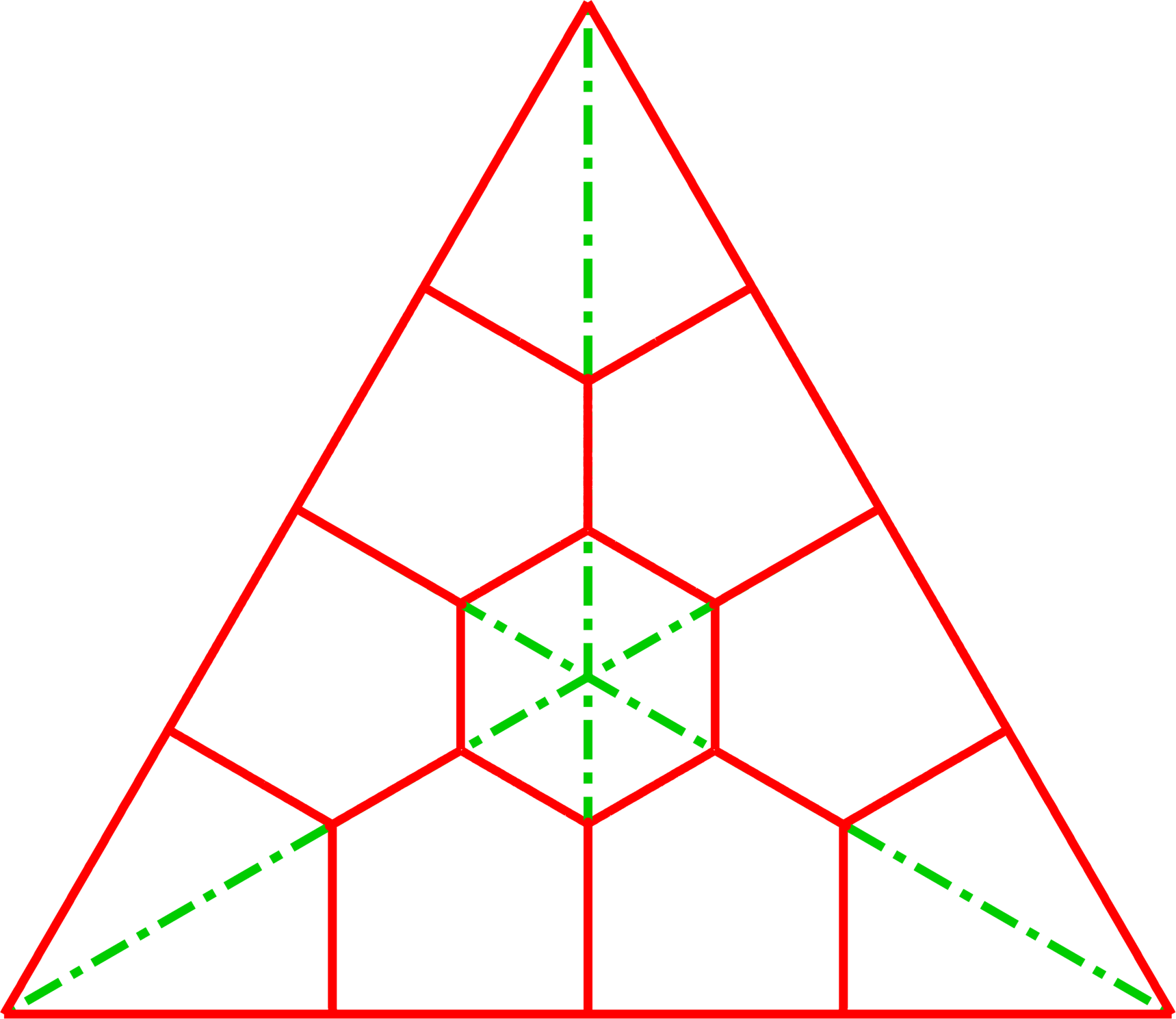}\\
$k=3$ & $k=5$ & $k=6$ & $k=7$ & $k=8$ & $k=9$ & $k=10$
\end{tabular}
\caption{Candidates to be minimal $k$-partition of the equilateral triangle for the max.}
\label{fig.triangleinf}
\end{figure}

\begin{figure}[h!]
\centering
\setlength{\tabcolsep}{1pt}
\begin{tabular}{ccccccc}
\includegraphics[width = 0.138\textwidth,angle=180,origin=c]{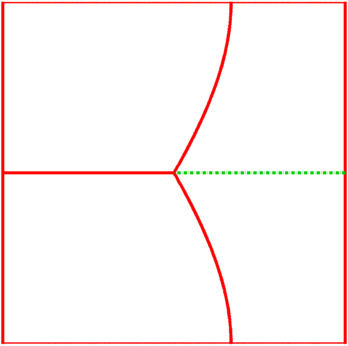}
&\includegraphics[width= 0.138\textwidth]{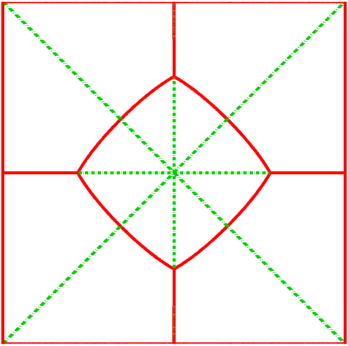}
&\includegraphics[width= 0.138\textwidth]{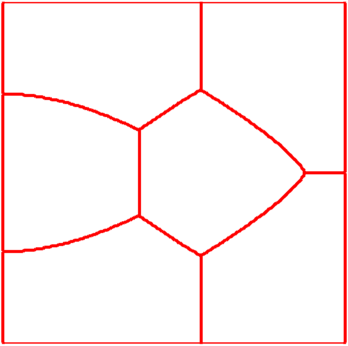}
&\includegraphics[width= 0.138\textwidth]{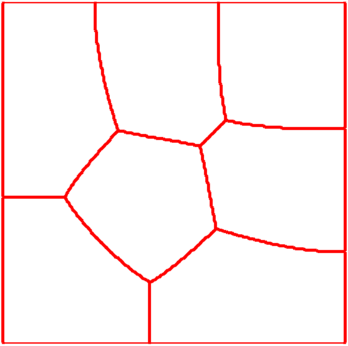}
&\includegraphics[width= 0.138\textwidth]{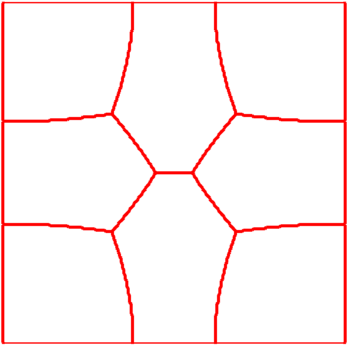}
&\includegraphics[width= 0.138\textwidth]{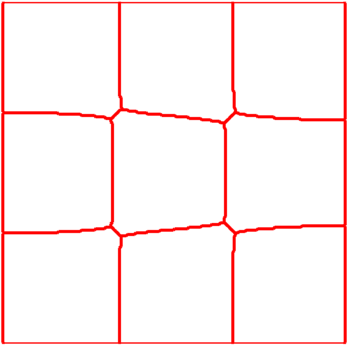}
&\includegraphics[width= 0.138\textwidth]{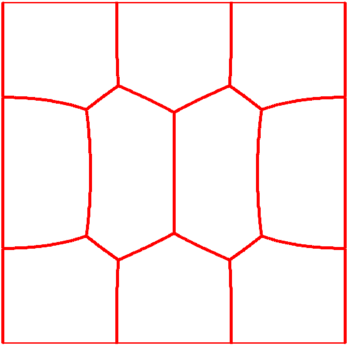}\\
$k=3$ & $k=5$ & $k=6$ & $k=7$ & $k=8$ & $k=9$ & $k=10$
\end{tabular}
\caption{Candidates to be minimal $k$-partition of the square for the max.}
\label{fig.squareinf}
\end{figure}

\begin{figure}[h!]
\centering
\setlength{\tabcolsep}{1pt}
\begin{tabular}{ccccccc}
\includegraphics[width=0.138\textwidth]{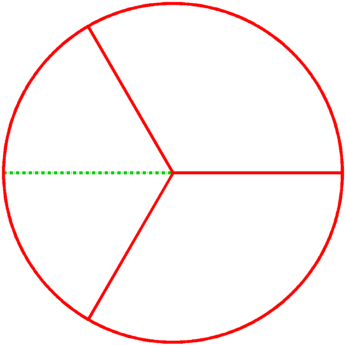}
& \includegraphics[width=0.138\textwidth]{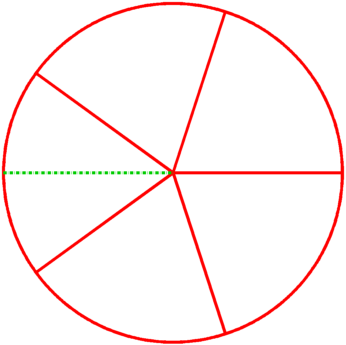}
& \includegraphics[width=0.138\textwidth]{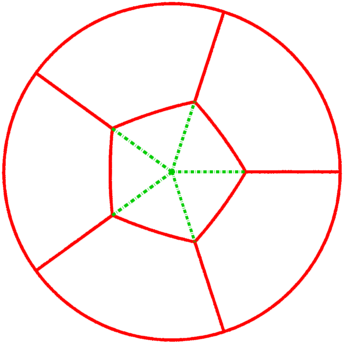}
& \includegraphics[width=0.138\textwidth]{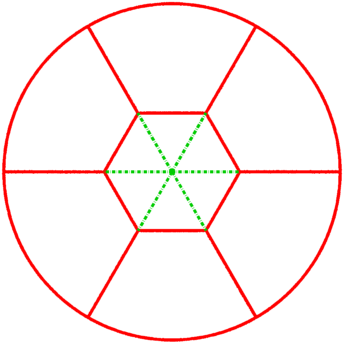}
& \includegraphics[width=0.138\textwidth]{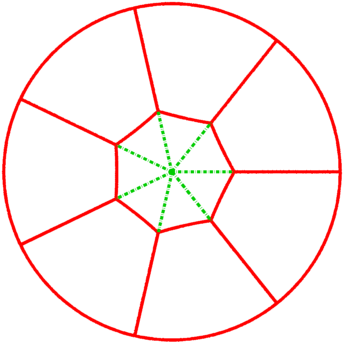}
& \includegraphics[width=0.138\textwidth]{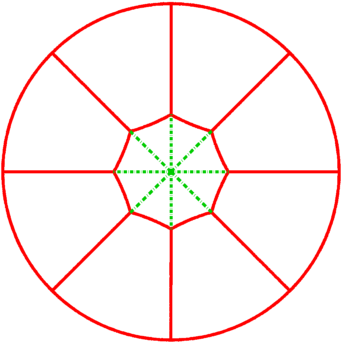}
&\includegraphics[width =0.138\textwidth]{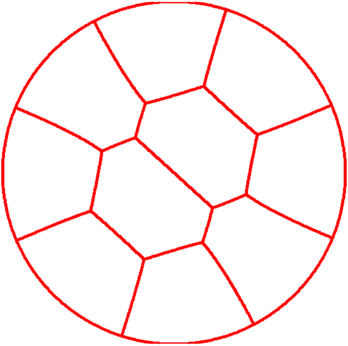}\\
$k=3$ & $k=5$ & $k=6$ & $k=7$ & $k=8$ & $k=9$ & $k=10$
\end{tabular}
\caption{Candidates to be minimal $k$-partitions of the disk for the max.}
\label{fig.disqueinf}
\end{figure}

\section{Optimal partitions for the max as candidates for the sum\label{sec.compar}} 
The aim of this section is to analyze if the optimal partitions obtained numerically for the max can be optimal for the sum.
Let us recall a criterion established  in \cite{HHO10} which prevents a partition optimal for the max to be optimal for the sum. 
\begin{prop}\label{l2norm}
Let ${\mathcal D}=(D_1,D_{2})$ be a minimal $2$-partition for the max and $\varphi_{2}$ be a second eigenfunction  of the Dirichlet-Laplacian on $\Omega$ having $ D_1$ and $ D_2$ as nodal domains. 
\begin{equation}\label{eq.criter}
\mbox{Suppose that }
\int_{D_1} |\varphi_{2}|^2 \neq \int_{D_2}|\varphi_{2}|^2,\qquad \mbox{ then }\quad
\mathfrak L_{2,1}(\Omega) < \mathfrak L_{2} (\Omega).\hfill\ 
\end{equation}
\end{prop}
Since any minimal $2$-partition for the max is nodal, the previous criteria can be generalized by considering two neighbors $D_{1},D_{2}$ of an optimal partition for the max. 
Note that this criterion gives no information if we apply it to an optimal partition for the max which is composed of congruent domains. 
It seems to be the case when $k=2,4$ for the square, $k=2,3,4,5$ for the disk and $k=3,4$ for the equilateral triangle  (see Figures~\ref{fig.PartNod}--\ref{fig.disqueinf}).

We know that partitions minimal for the max are in fact equipartitions (see Theorem~\ref{thm.ex}). Yet the Dirichlet-Neumann approach produces equipartition whereas with the iterative method, the first eigenvalues on the subdomains of the partition obtained numerically are close but not rigorously equal. 
Thus we consider from now in this section the partitions obtained with the Dirichlet-Neumann approach. 
In such situation, Proposition \ref{l2norm} can be adapted as follows.
\begin{prop}
Let $\mathcal{D} = (D_i)_{1\leq i \leq k}$ be a minimal $k$-partition for the max. We assume that $\mathcal D$ is the nodal partition of an eigenfunction $\varphi$ of a mixed Dirichlet-Neumann problem. Then this partition is not minimal for the sum if we can find two neighboring cells on which the $L^2$ norms of the eigenfunction $\varphi$ are different. 
\label{l2normappl}
\end{prop}
\begin{proof}
It is not difficult to see that the previous result is a simple consequence of Proposition \ref{l2norm}. Indeed, if $D_1,D_2$ are two neighbors of a nodal partition associated to a mixed Dirichlet-Neumann problem, then $(D_1,D_2)$ is a minimal partition for the max on ${\rm Int}(\overline{D_1} \cup \overline{D_2})$. If this is not the case, then $\mathcal{D}$ would not be minimal. Furthermore, the restriction of the eigenfunction $\varphi$ to ${\rm Int}(\overline{D_1} \cup \overline{D_2})$ is an eigenfunction for the second eigenvalue on the same domain ${\rm Int}(\overline{D_1} \cup \overline{D_2})$. Thus, if the $L^2$ norms of $\varphi$ are different on $D_1$ and $D_2$, we may apply Proposition \ref{l2norm} to conclude that $(D_1,D_2)$ is not optimal for the sum on ${\rm Int}(\overline{D_1} \cup \overline{D_2})$. This immediately implies that $\mathcal{D}$ is not optimal for the sum on $\Omega$. 
\end{proof}
 
Let us recall that the mixed Dirichlet-Neumann method consists of expressing the minimal partition for the max as a nodal partition corresponding to a mixed problem. When we look for symmetric partitions, we consider a mixed Dirichlet-Neumann problem on a reduced domain and we search the eigenfunction which realizes the minimal energy in the reduced domain, taking care that the associated nodal partition has the desired structure. In the following we use the results of Proposition \ref{l2normappl} which allows us to deduce if the partition can be a candidate for the sum by looking at the $L^2$ norms of the eigenfunction of the mixed problem on the subdomains. This simplifies the computation, since we are not forced to extract each pair of neighbor domains and compute a new second eigenfunction on these domains. Given the eigenfunction of the mixed problem we can identify each of the subdomains by looking at its sign of restrictions to certain rectangles: $\{\varphi>0 \text{ or }\varphi<0\} \cap \{x \in [a,b]\}\cap \{y \in [c,d]\}$. Once we have computed the $L^2$ norms on subdomains corresponding to the mixed problem, we may multiply the norms corresponding to domains cut by symmetry axes by a corresponding factor, to find the $L^2$ norms on the initial partition.  The computation of the $L^2$ norm is made either in FreeFem++ \cite{freefem} or in Melina \cite{melina}.

Let us now examine the candidates obtained with the Dirichlet-Neumann approach: $k \in \{6,7,8,9\}$ for the disk, $k \in \{3,5\}$ for the square and $k \in \{5,6,8,10\}$ for the equilateral triangle. These configurations correspond in Figures~\ref{fig.triangleinf}, \ref{fig.squareinf} and \ref{fig.disqueinf} with green straight doted lines for which the subdomains are not congruent. We present the $L^2$ norms on each one of the subdomains. The situation is clear in the case of the disk and the square where we have only two types of domains: a single domain $D_{1}$ (the interior domain in the case of the square with $k=5$ or the disk) and a domain $D_{2}$ repeated several times by some symmetry (the exterior domains for the square with $k=5$ or the disk). Table \ref{tab.cardis} shows that the $L^2$ norm of the eigenfunctions on the different subdomains are not equal. Thus, in the case of the disk and $k=5,6,7,8$ or the square and $k=3,5$, if the partitions of Figures~\ref{fig.squareinf} and \ref{fig.disqueinf} are optimal for the max, they are not optimal for the sum.
\begin{table}[h!t]
\begin{center}
\begin{tabular}{|c|c|c|c|c|c|c|c|}
\hline
$\Omega$ & \multicolumn{2}{c|}{Square} & \multicolumn{4}{c|}{Disk}\\
\hline
$k$ & 3 & 5 & 6 & 7 & 8 & 9\\
\hline
$\int_{D_{1}}|\varphi|^2$ &  0.51 & 1.12 & 1.12 & 0.88 & 0.58 & 0.29\\
$\int_{D_{2}}|\varphi|^2$ & 0.75 & 0.72 & 0.78 & 0.85 & 0.92 & 0.96\\
\hline
\end{tabular}
\end{center}
\caption{$L^2$-norm on the square and the disk.\label{tab.cardis}}
\end{table}

In the case of the equilateral triangle we may have up to $5$ different domains in the partition. We denote these subdomains $D_1,D_2,...$ by starting from a vertex of the triangle and going along the sides in a clockwise rotation sense. In the following we write the values of the $L^2$ norms on each one of these domains for $k=2,5,8,6,10$:\renewcommand\labelitemi{{\boldmath$\cdot$}}
\begin{itemize}[leftmargin=*]
\item $k=2$:  $\int_{D_1} |\varphi|^2 = 0.49,\ \int_{D_2}|\varphi|^2 = 0.51$;
\item $k=5$: $\int_{D_1}|\varphi|^2 = 0.45,\ \int_{D_2} |\varphi|^2 = 0.44,\ \int_{D_3}|\varphi|^2 =0.34$;
\item $k=8$: $\int_{D_1}|\varphi|^2 = 0.31,\ \int_{D_2}|\varphi|^2 = 0.30,\ \int_{D_3} |\varphi|^2=0.22,\ \int_{D_4}|\varphi|^2 = 0.17,\ \int_{D_5}|\varphi|^2 =0.23$;
\item $k=6$: $\int_{D_1}|\varphi|^2  = 0.500,\ \int_{D_2} |\varphi|^2 = 0.500$;
\item $k=10$: $\int_{D_1}|\varphi|^2 = 0.599,\ \int_{D_2}|\varphi|^2 = 0.600,\ \int_{D_3}|\varphi|^2 = 0.604$. 
\end{itemize}
In each of these cases $k\in\{2,5,8\}$, we find at least one pair of adjacent domains which have different $L^2$ norm, so these candidates are not optimal for the sum. Note that the case $k=2$ has already been treated in \cite{HHO10}.
In the cases $k\in \{6,10\}$ we do not find adjacent domains with significantly different $L^2$ norms, so the criterion does not apply. Moreover, a parametric study was done in \cite{BBN16}, assuming the borders of the cells are polygonal domains. This study shows that the partitions for the sum and for the max likely coincide in these cases.

\section{Candidates for the sum\label{sec.sum}}

\subsection{Iterative method}
Let us apply the iterative method to exhibit candidates for the sum. In Figures~\ref{fig.equisummax}, \ref{fig.disksummax} and \ref{fig.carresummax}, we compare these candidates for the sum (represented in blue) with the candidates for the max (in red) obtained by the iterative method with penalization. 
\begin{figure}[h!t]
\centering
\setlength{\tabcolsep}{1pt}
\begin{tabular}{cccc}
\includegraphics[width=0.19\textwidth]{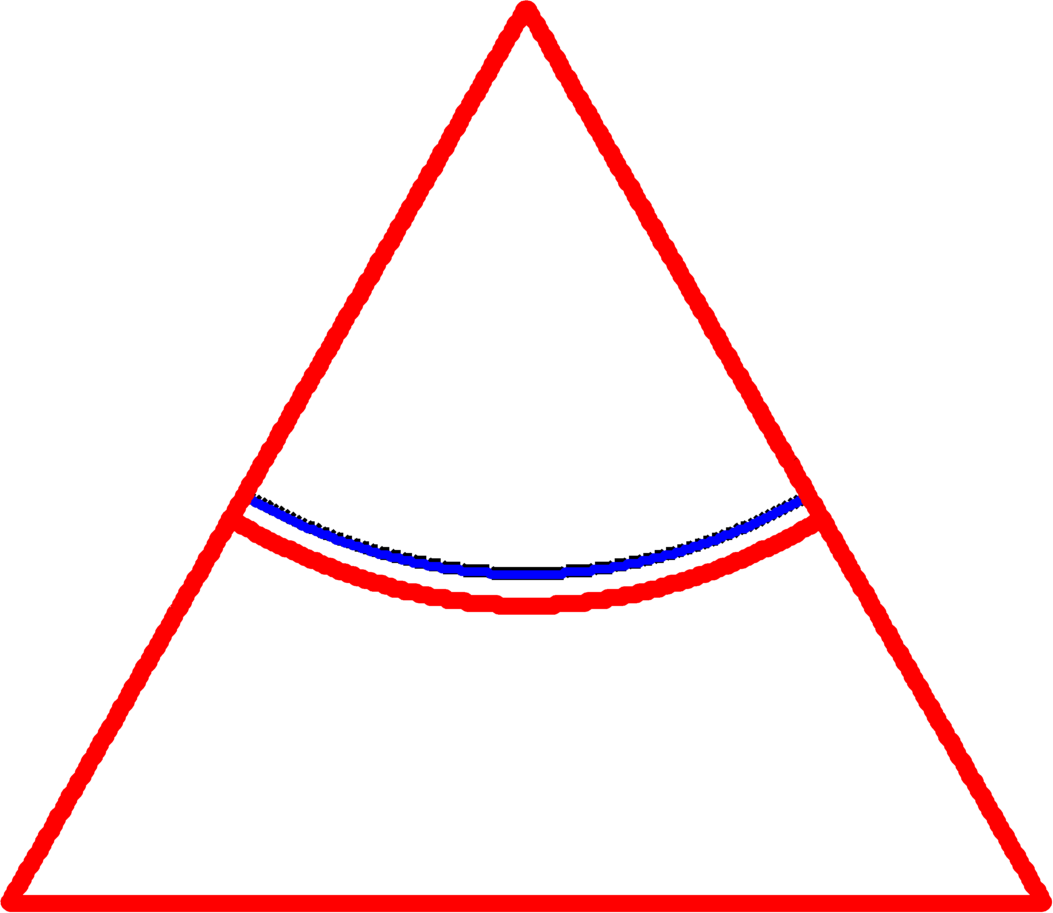}
& \includegraphics[width=0.19\textwidth]{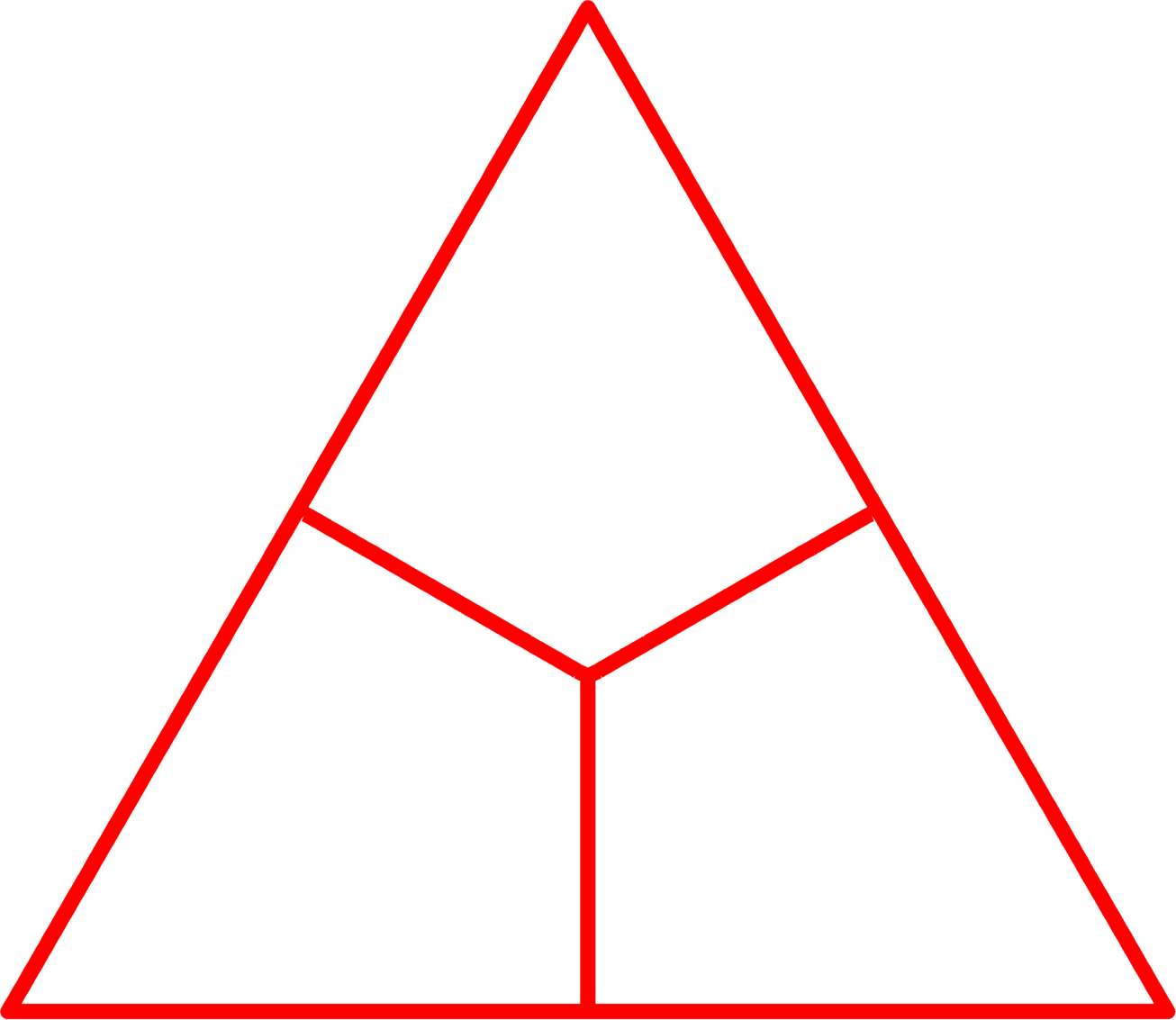}
& \includegraphics[width=0.19\textwidth]{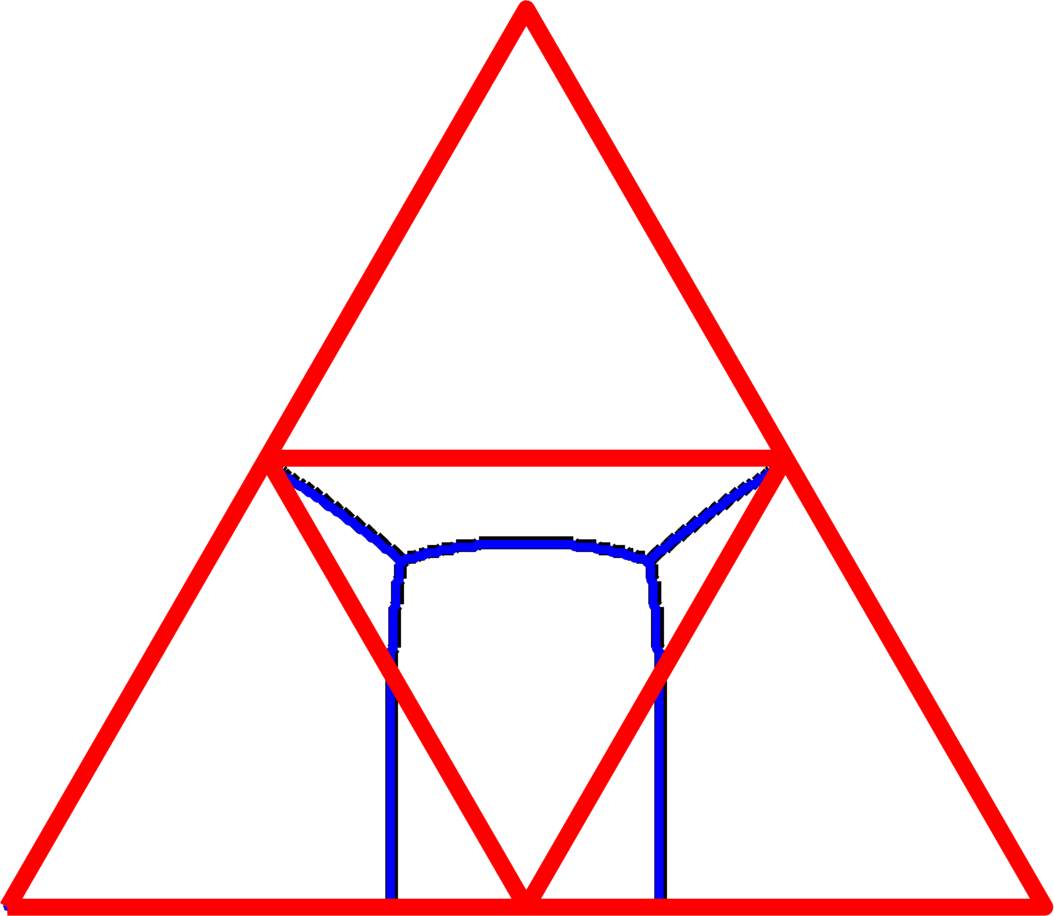}
& \includegraphics[width=0.19\textwidth]{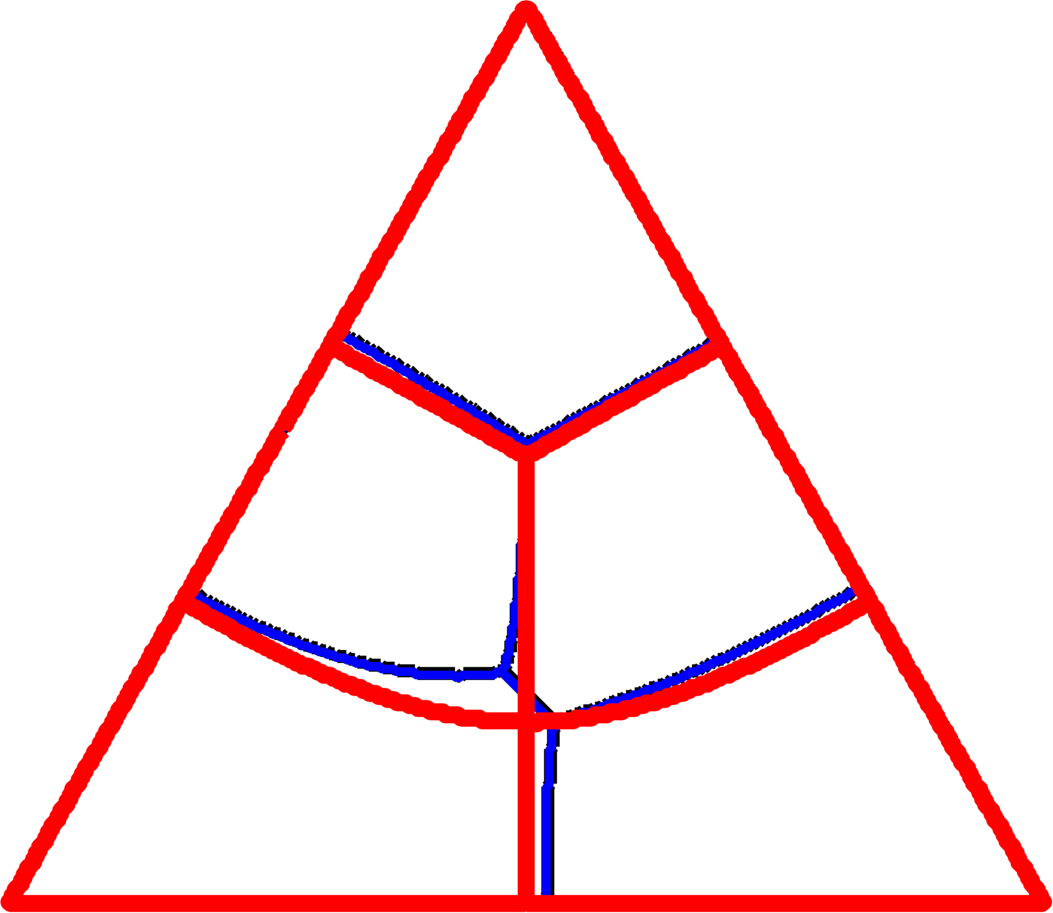}
\\
$k=2$ & $k=3$ & $k=4$ & $k=5$ \\
\end{tabular}
\begin{tabular}{ccccc}
\includegraphics[width=0.19\textwidth]{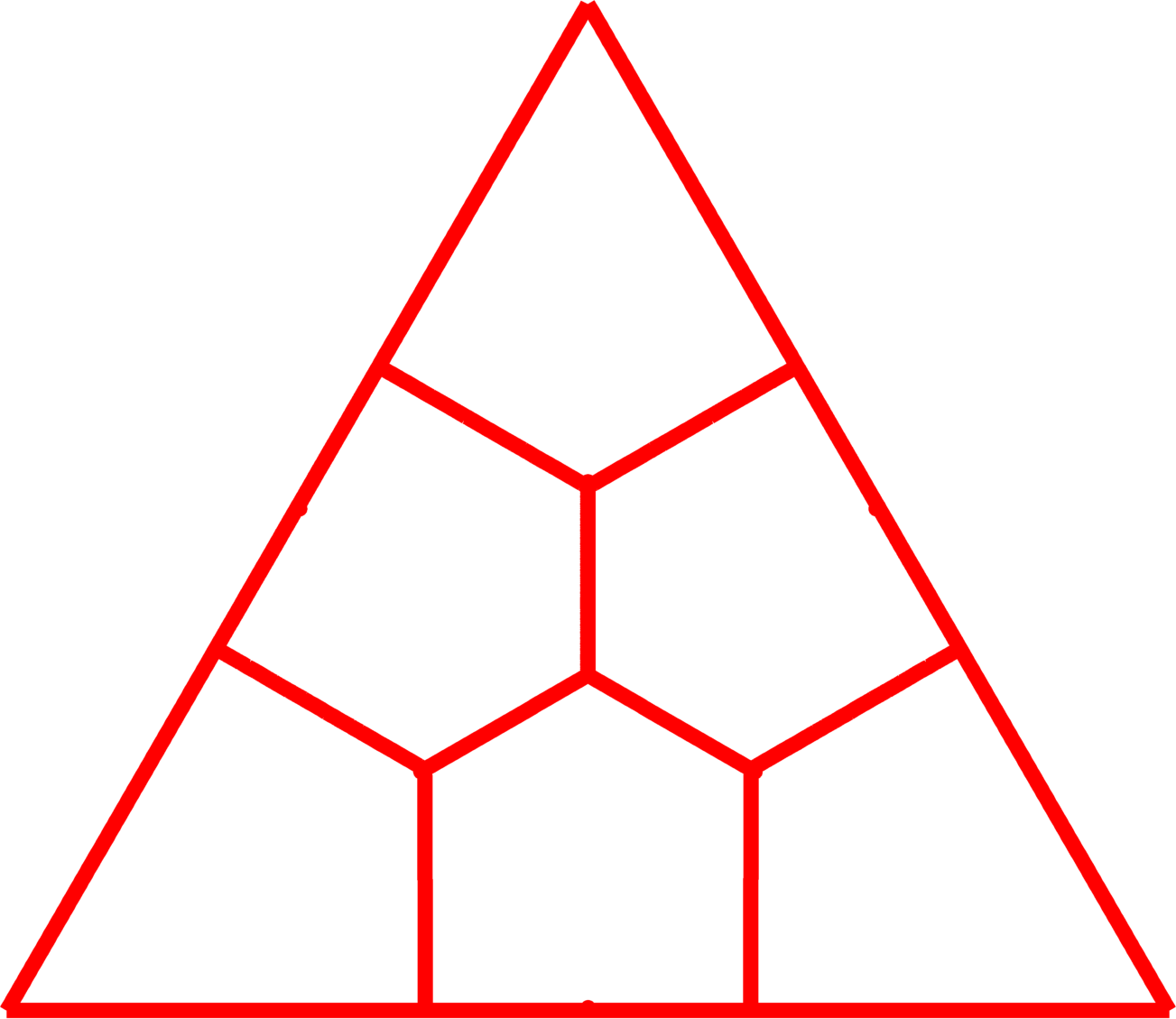}
& \includegraphics[width=0.19\textwidth]{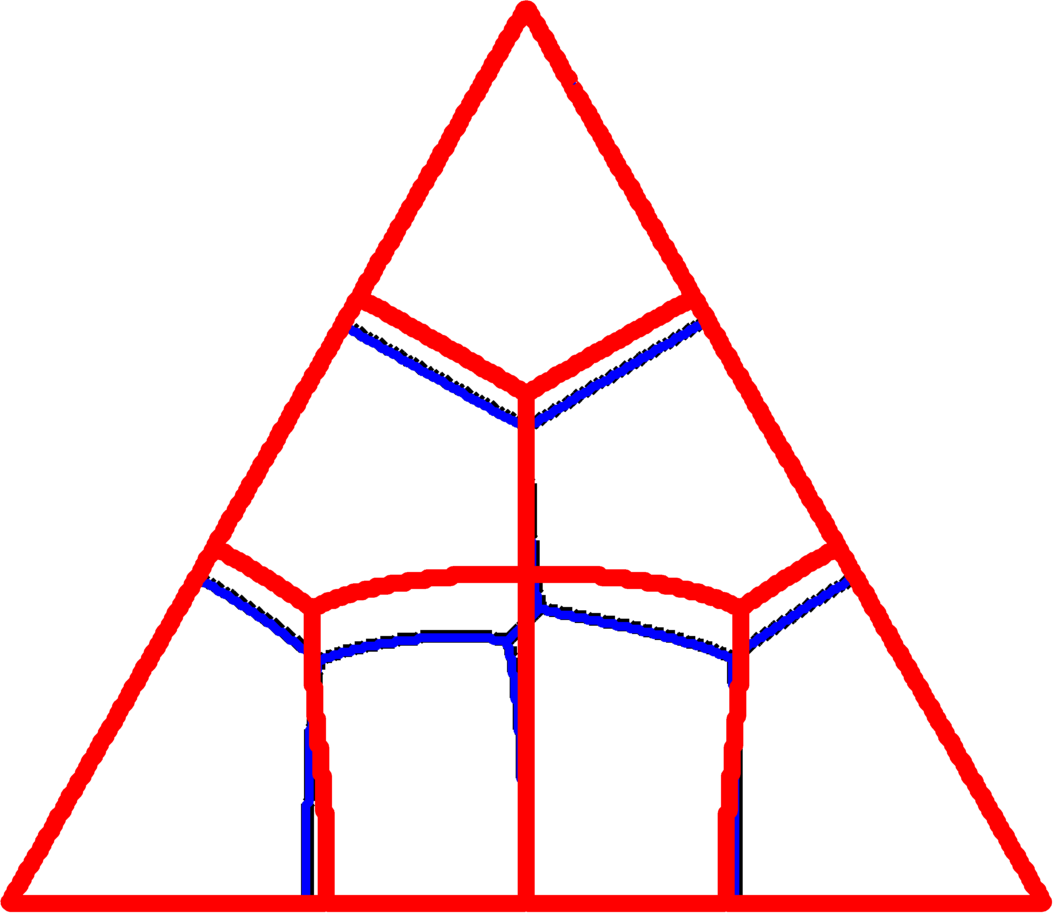}
& \includegraphics[width=0.19\textwidth]{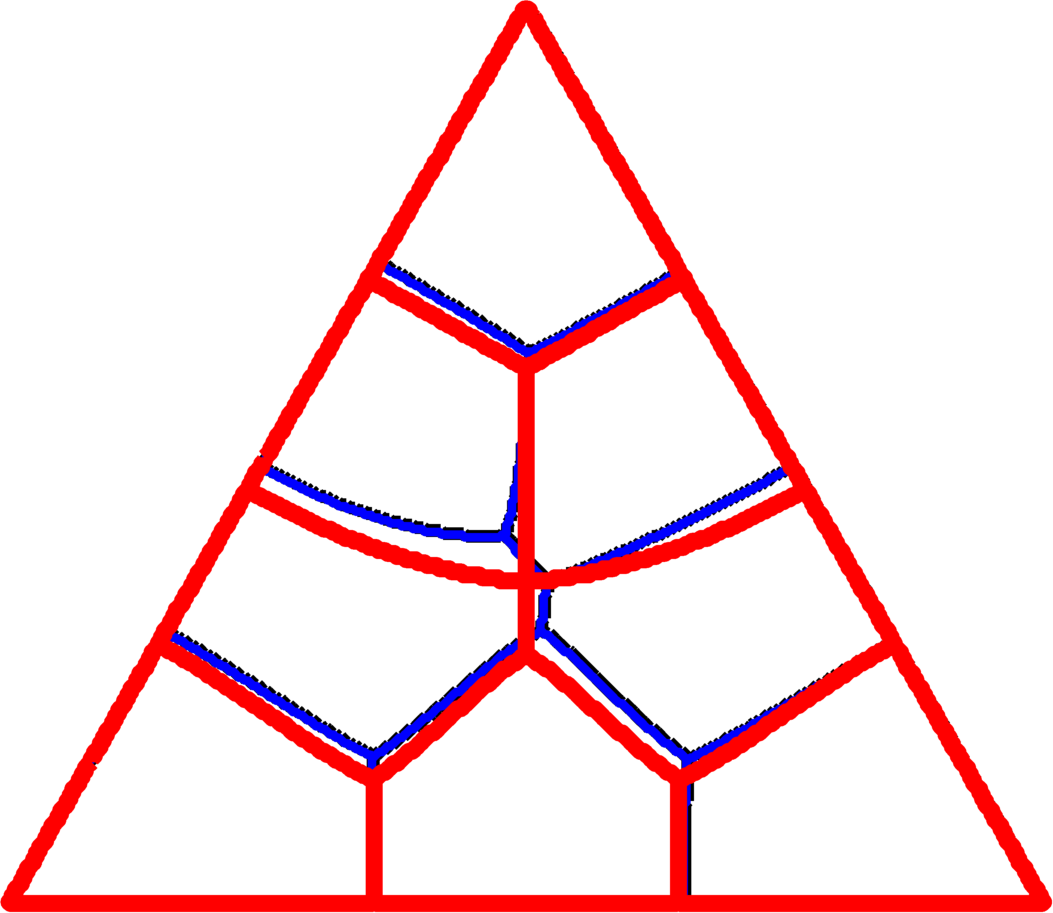}
&\includegraphics[width=0.19\textwidth]{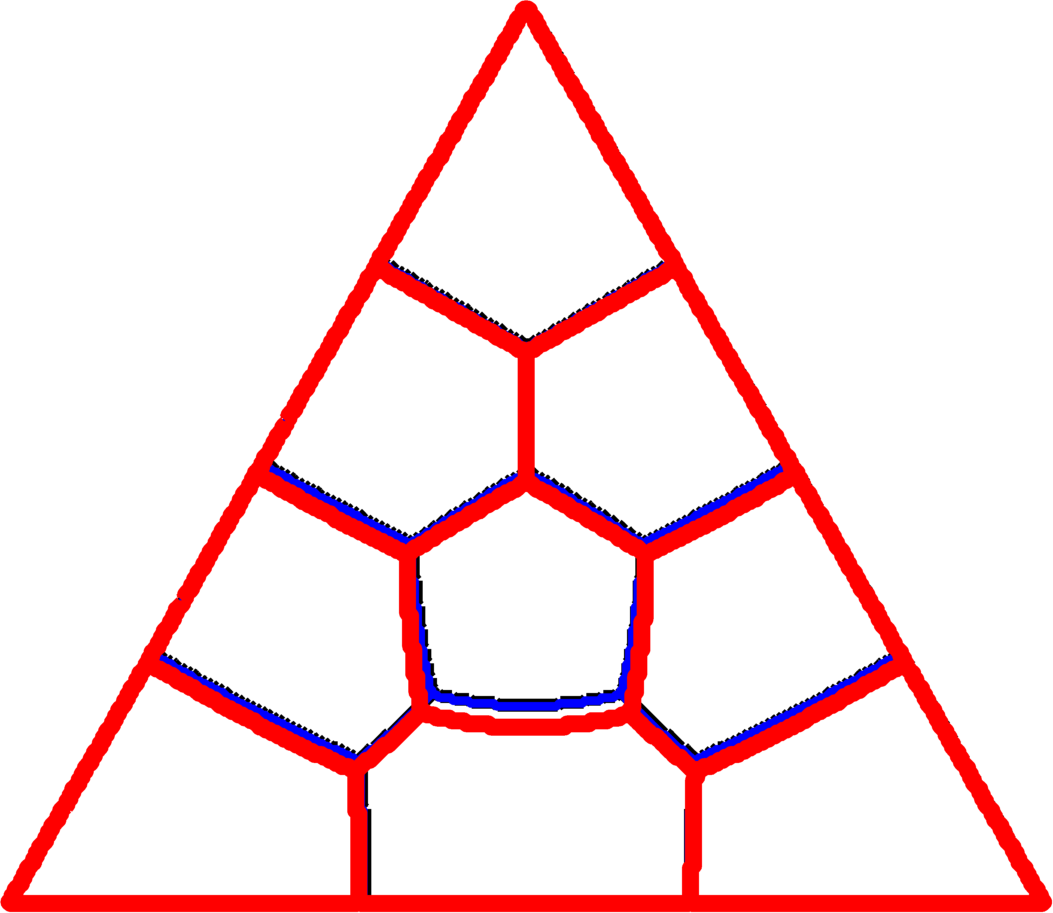}
&\includegraphics[width=0.19\textwidth]{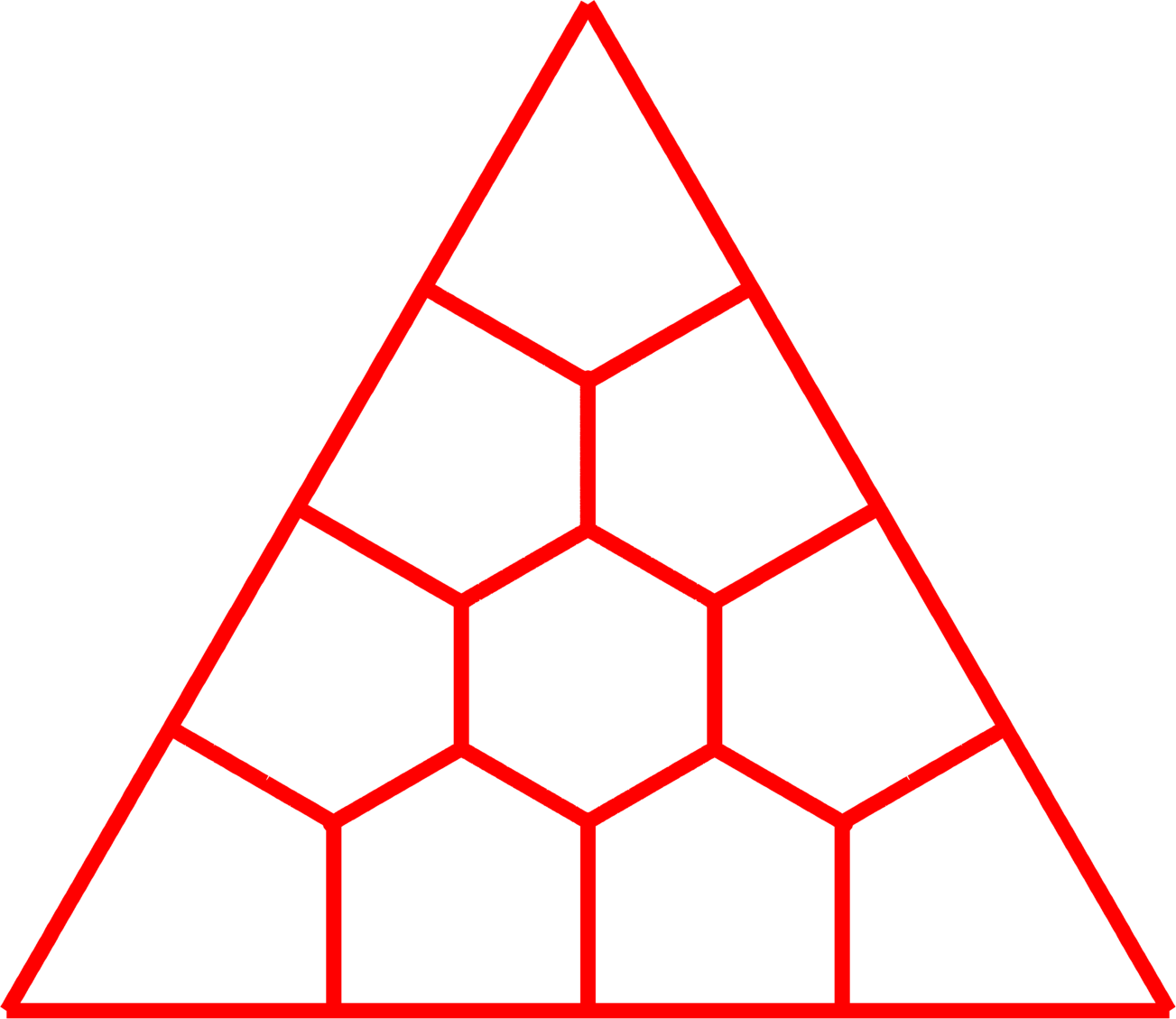} \\
$k=6$ & $k=7$ & $k=8$ & $k=9$ & $k=10$
\end{tabular}
\caption{Candidates for the max (red) and the sum (blue).\label{fig.equisummax}}
\end{figure}
\begin{figure}[h!t]
\centering
\setlength{\tabcolsep}{1pt}
\begin{tabular}{ccccccccc}
\includegraphics[width=0.19\textwidth]{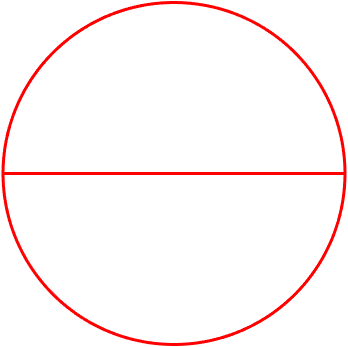}
& \includegraphics[width=0.19\textwidth]{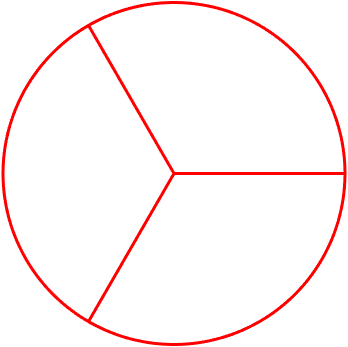}
& \includegraphics[width=0.19\textwidth]{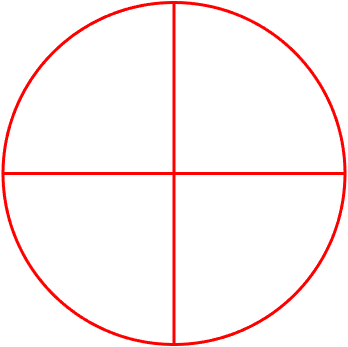}
& \includegraphics[width=0.19\textwidth]{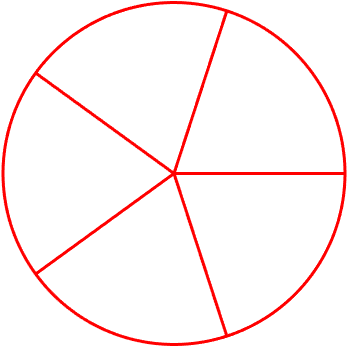}\\
$k=2$ & $k=3$ & $k=4$ &  $k=5$ \\
\end{tabular}
\begin{tabular}{ccccccccc}
\includegraphics[width=0.19\textwidth]{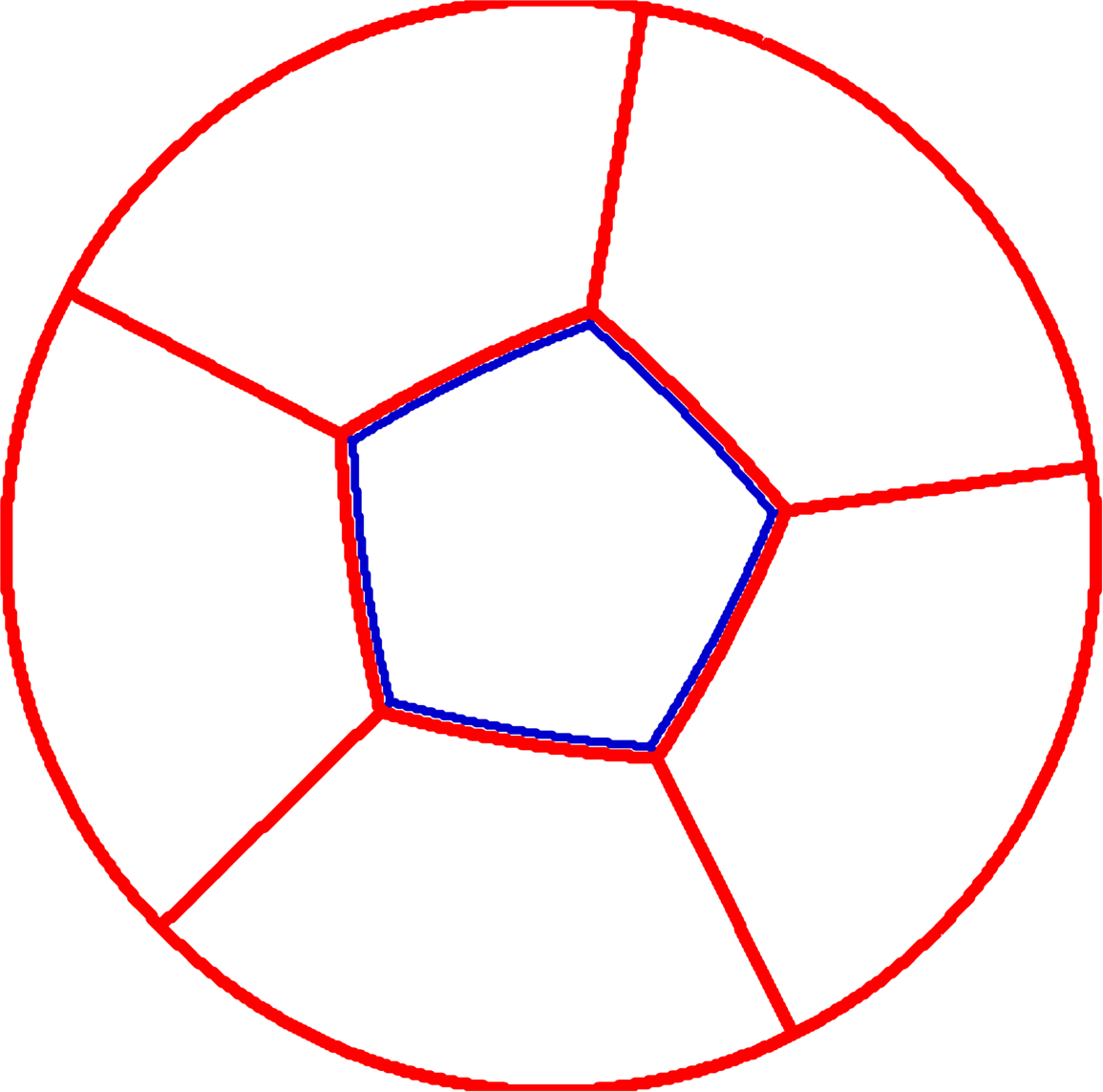}
&\includegraphics[width=0.19\textwidth]{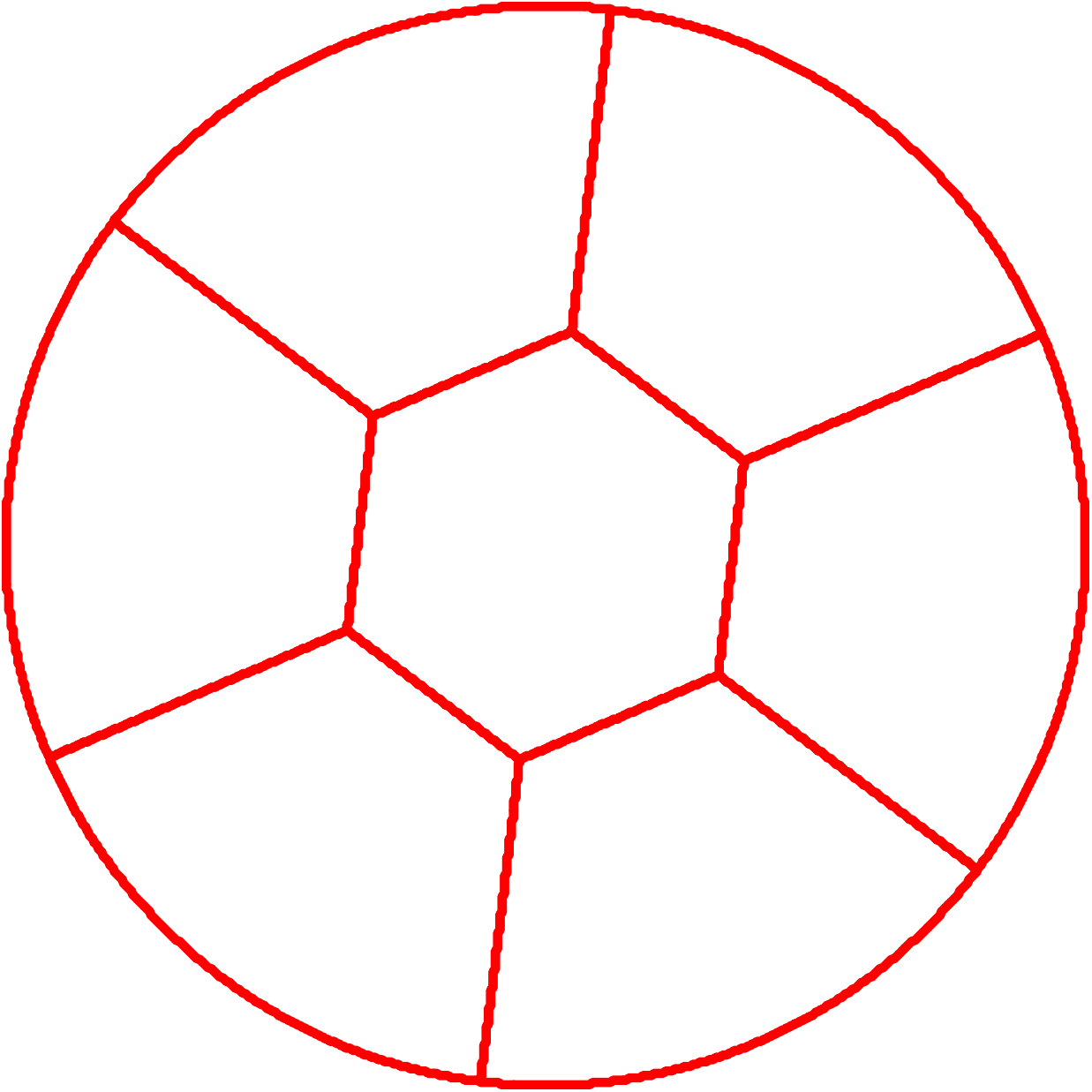}
& \includegraphics[width=0.19\textwidth]{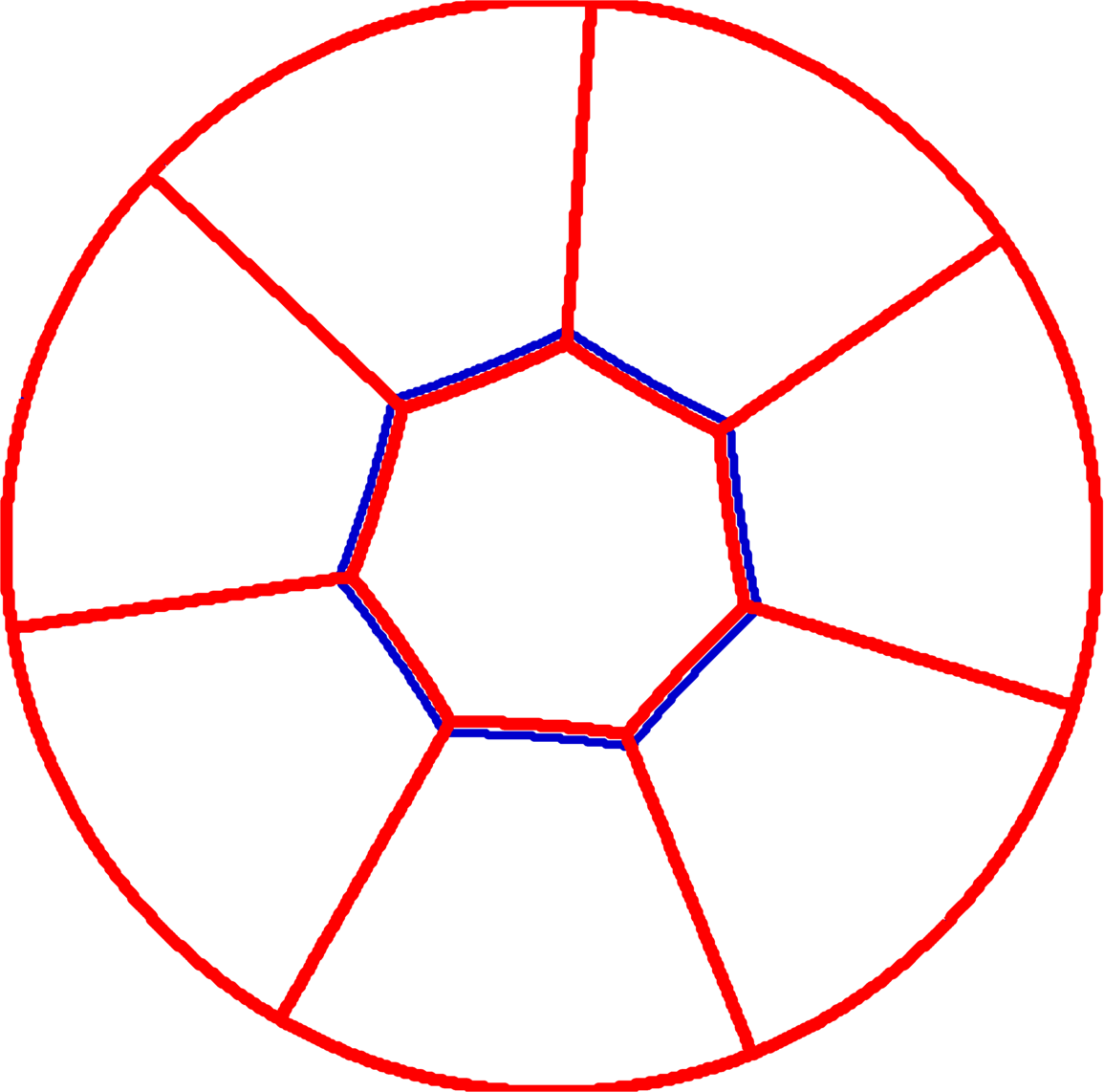}
& \includegraphics[width=0.19\textwidth]{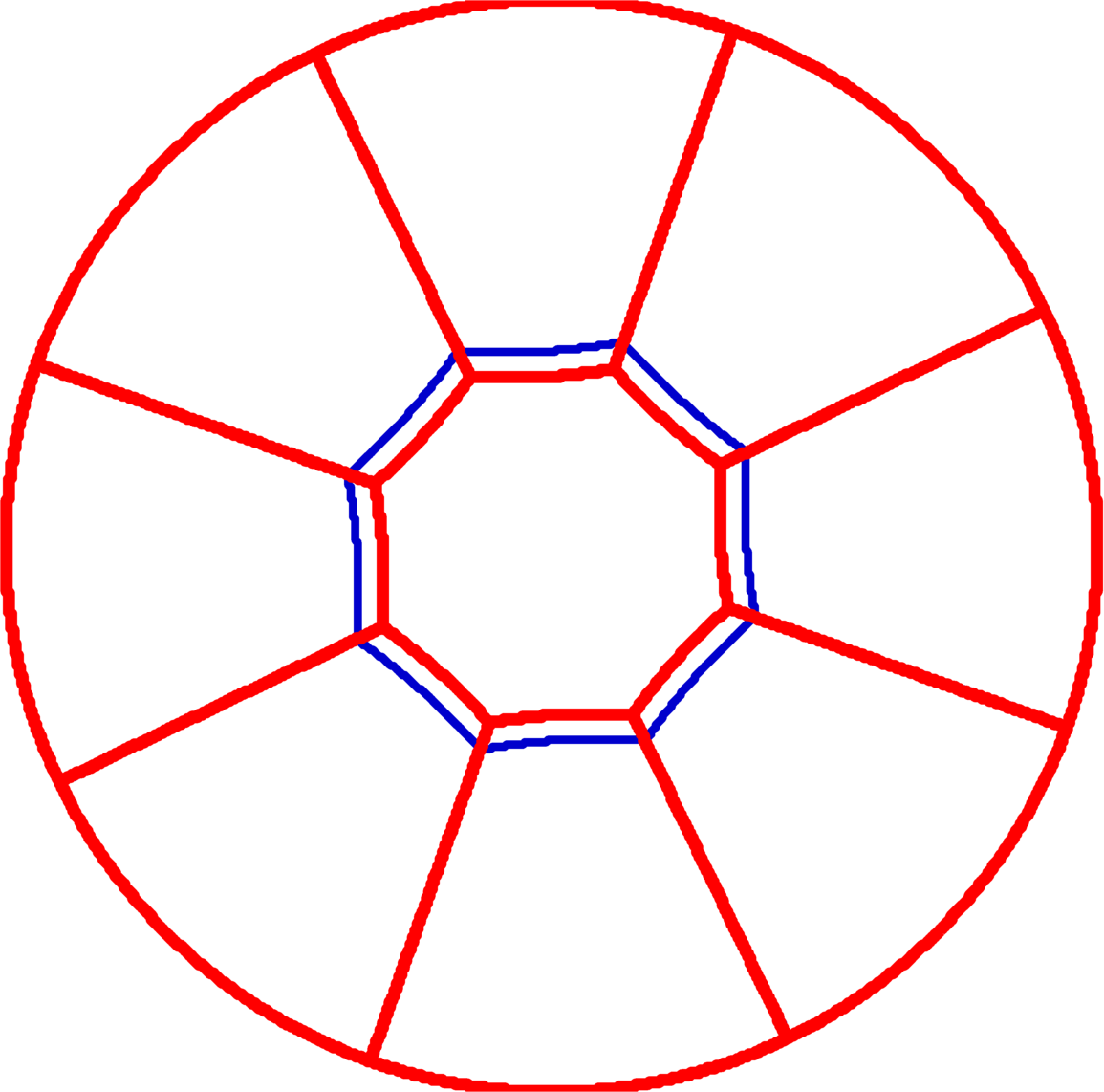}
& \includegraphics[width=0.19\textwidth]{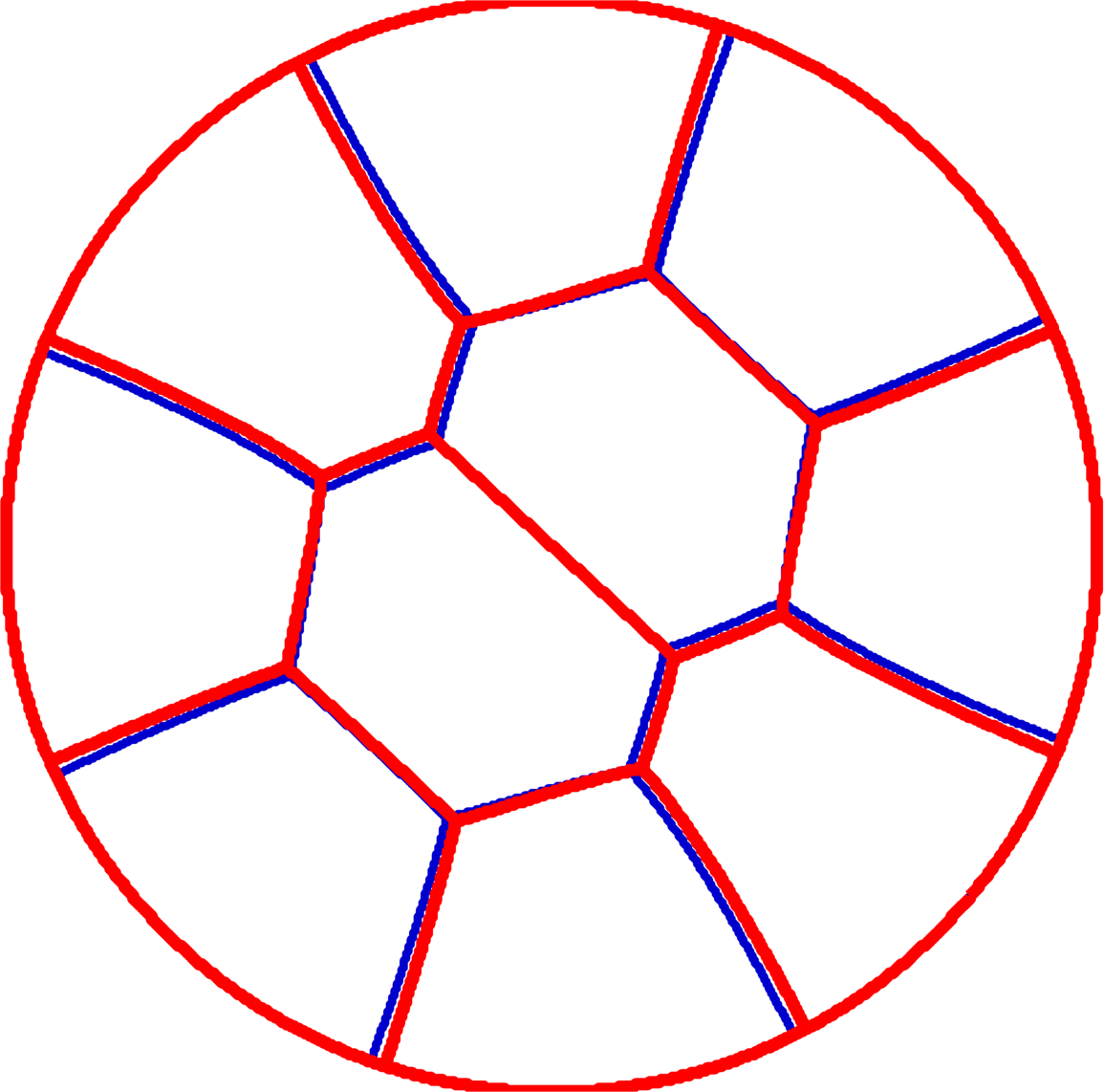}\\
$k=6$ & $k=7$ & $k=8$ & $k=9$ & $k=10$
\end{tabular}
\caption{Candidates for the max (red) and the sum (blue).\label{fig.disksummax}}
\end{figure}
\begin{figure}[h!t]
\centering
\setlength{\tabcolsep}{1pt}
\begin{tabular}{cccccccc}
\includegraphics[width=0.19\textwidth]{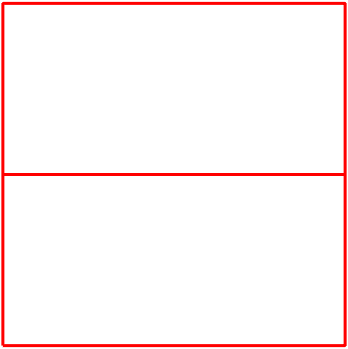}
& \includegraphics[width=0.19\textwidth,angle=270,origin=c]{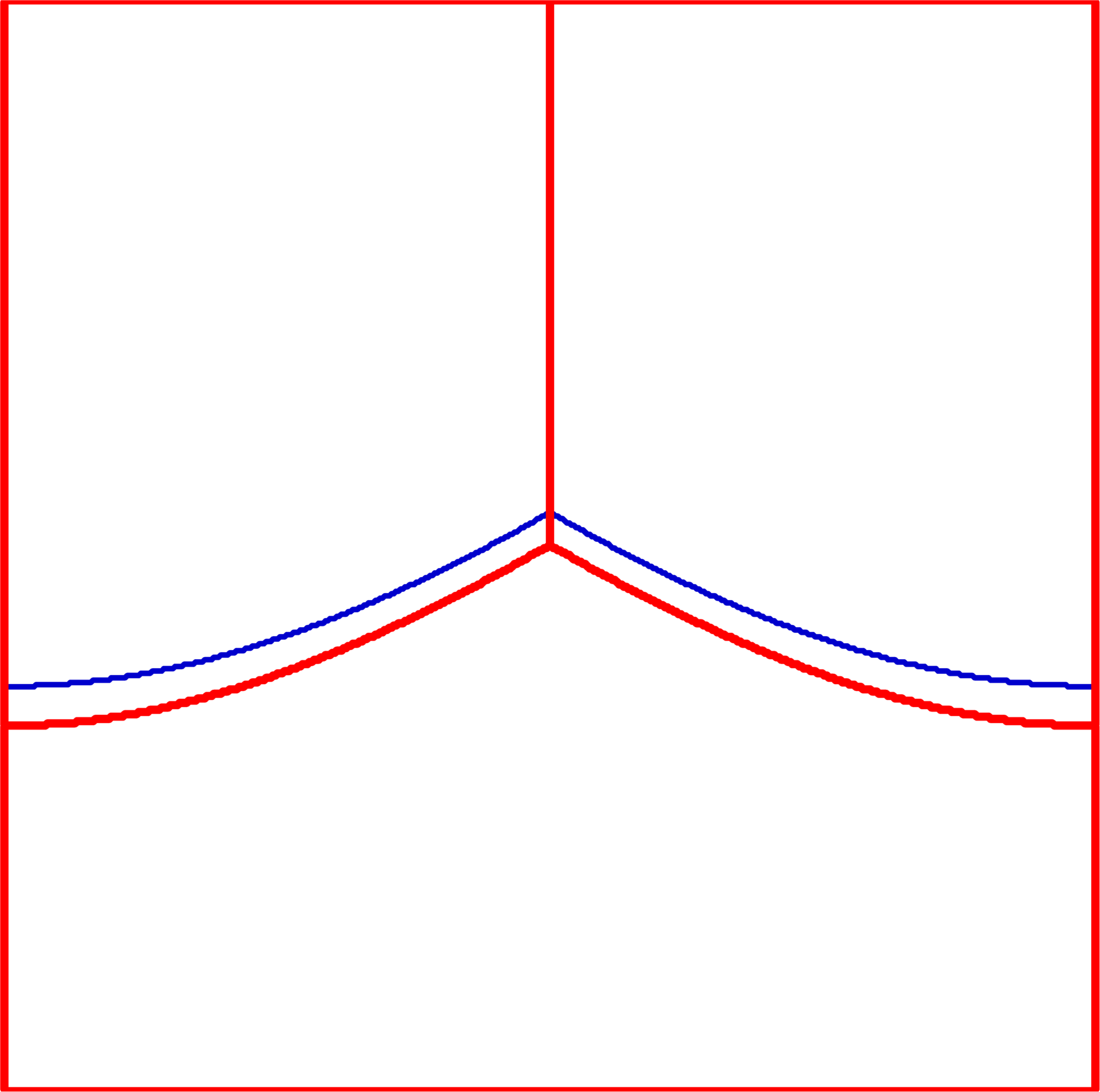}
& \includegraphics[width=0.19\textwidth]{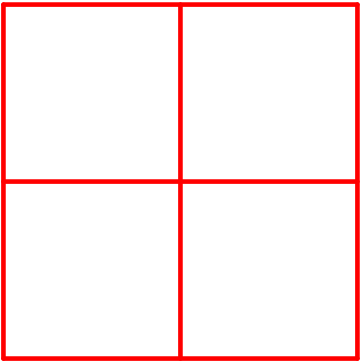}
& \includegraphics[width=0.19\textwidth]{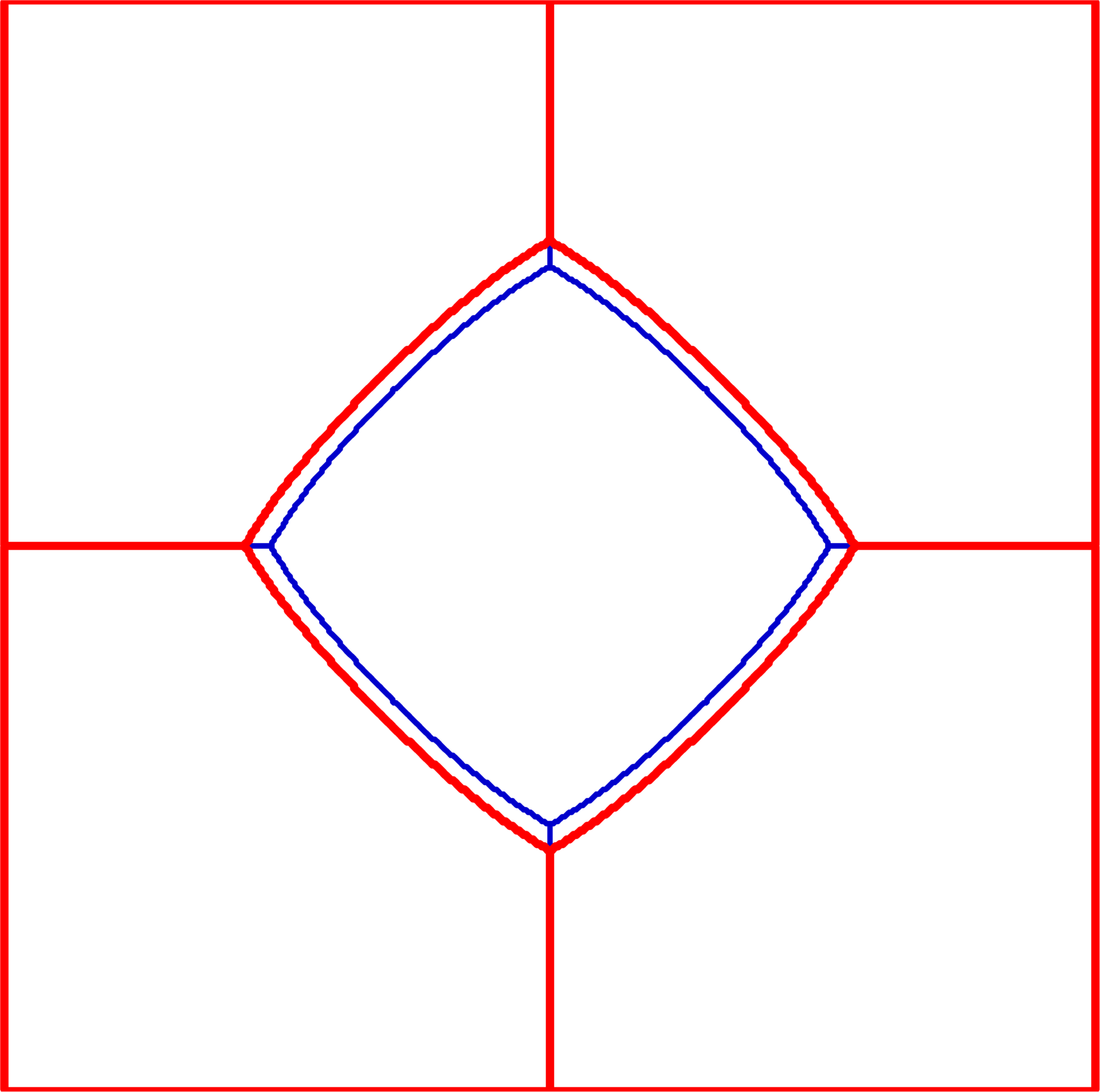}\\
$k=2$ & $k=3$ & $k=4$ & $k=5$\\
\end{tabular}
\begin{tabular}{ccccccccc}
\includegraphics[width=0.19\textwidth]{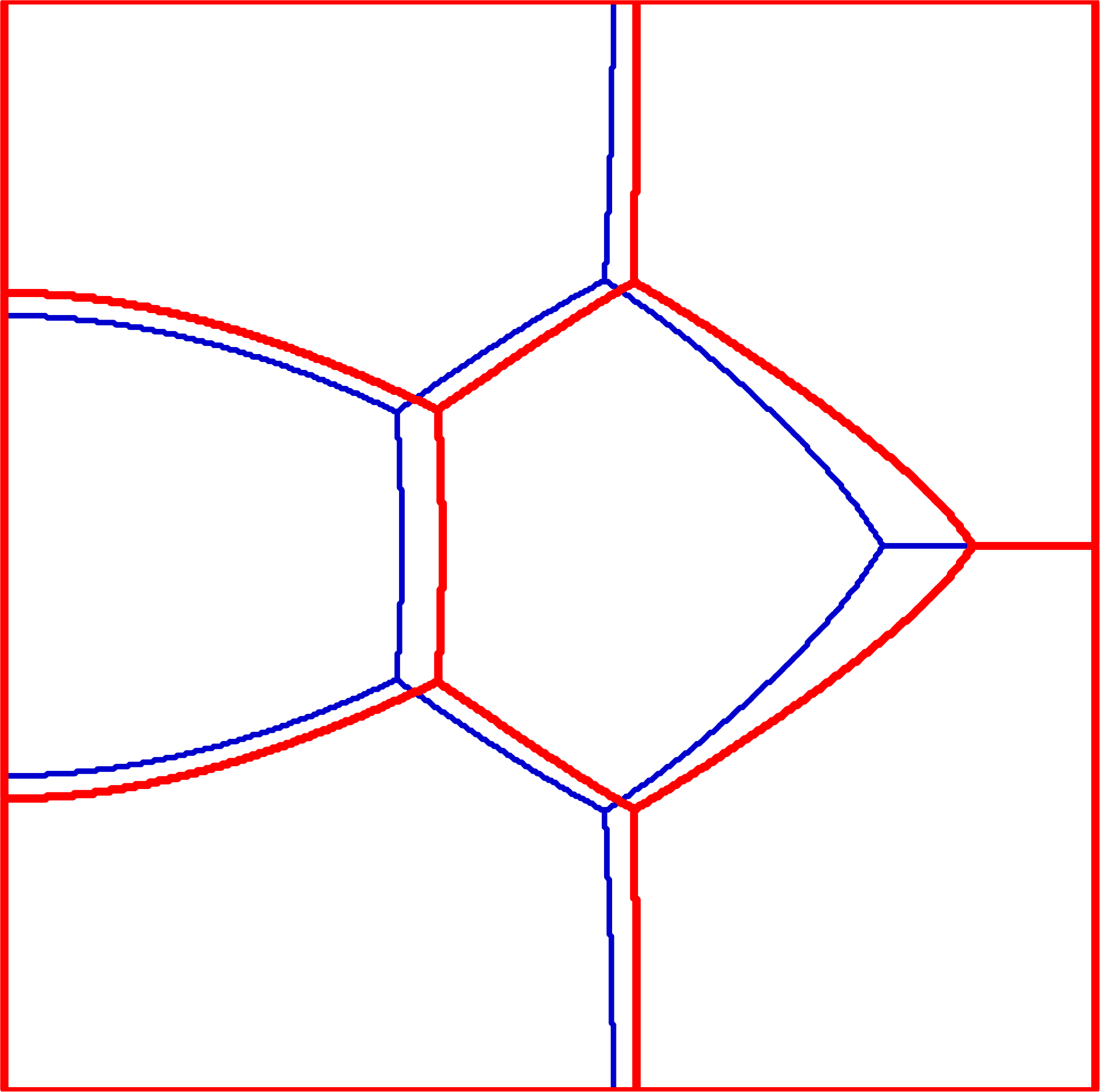}
& \includegraphics[width=0.19\textwidth]{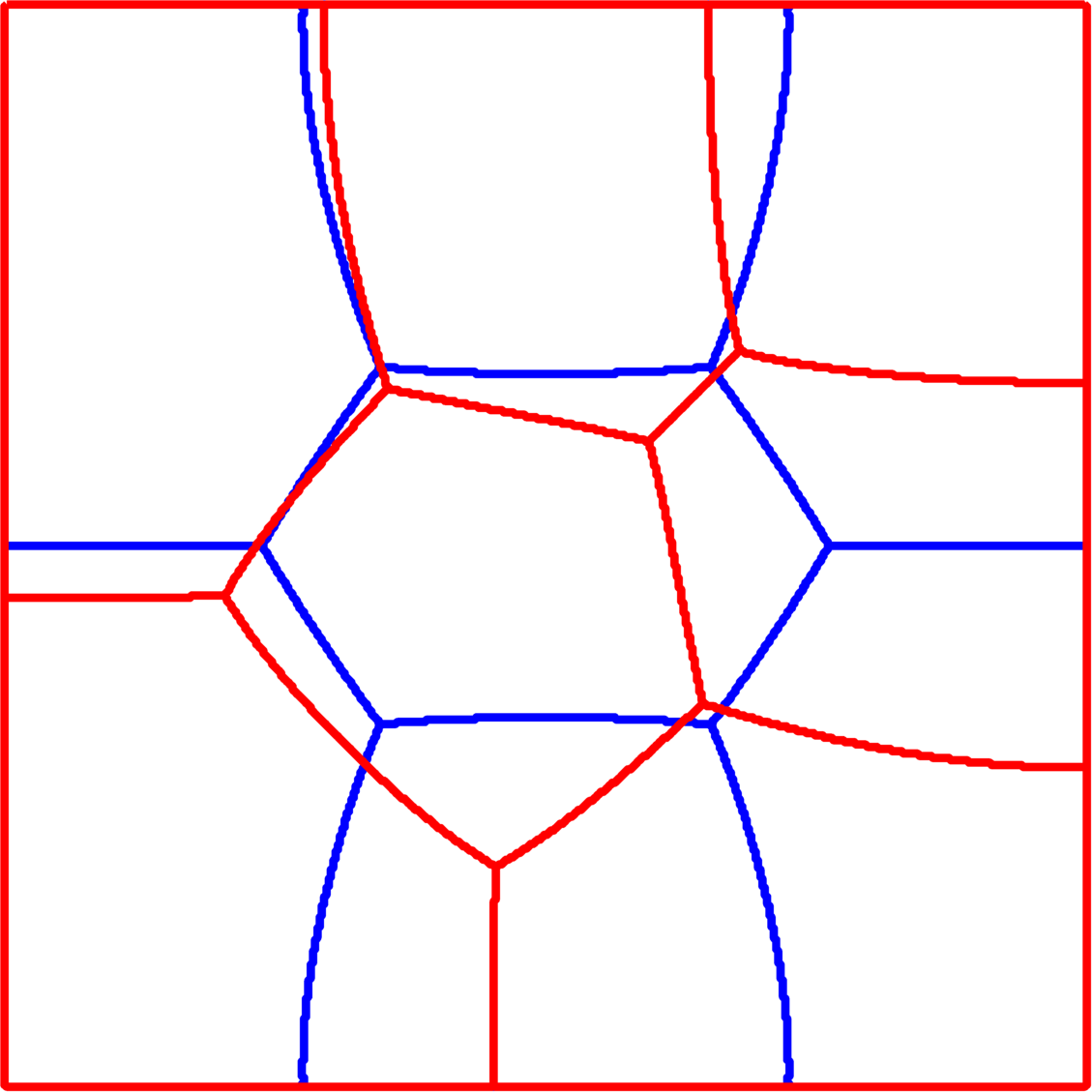}
& \includegraphics[width=0.19\textwidth]{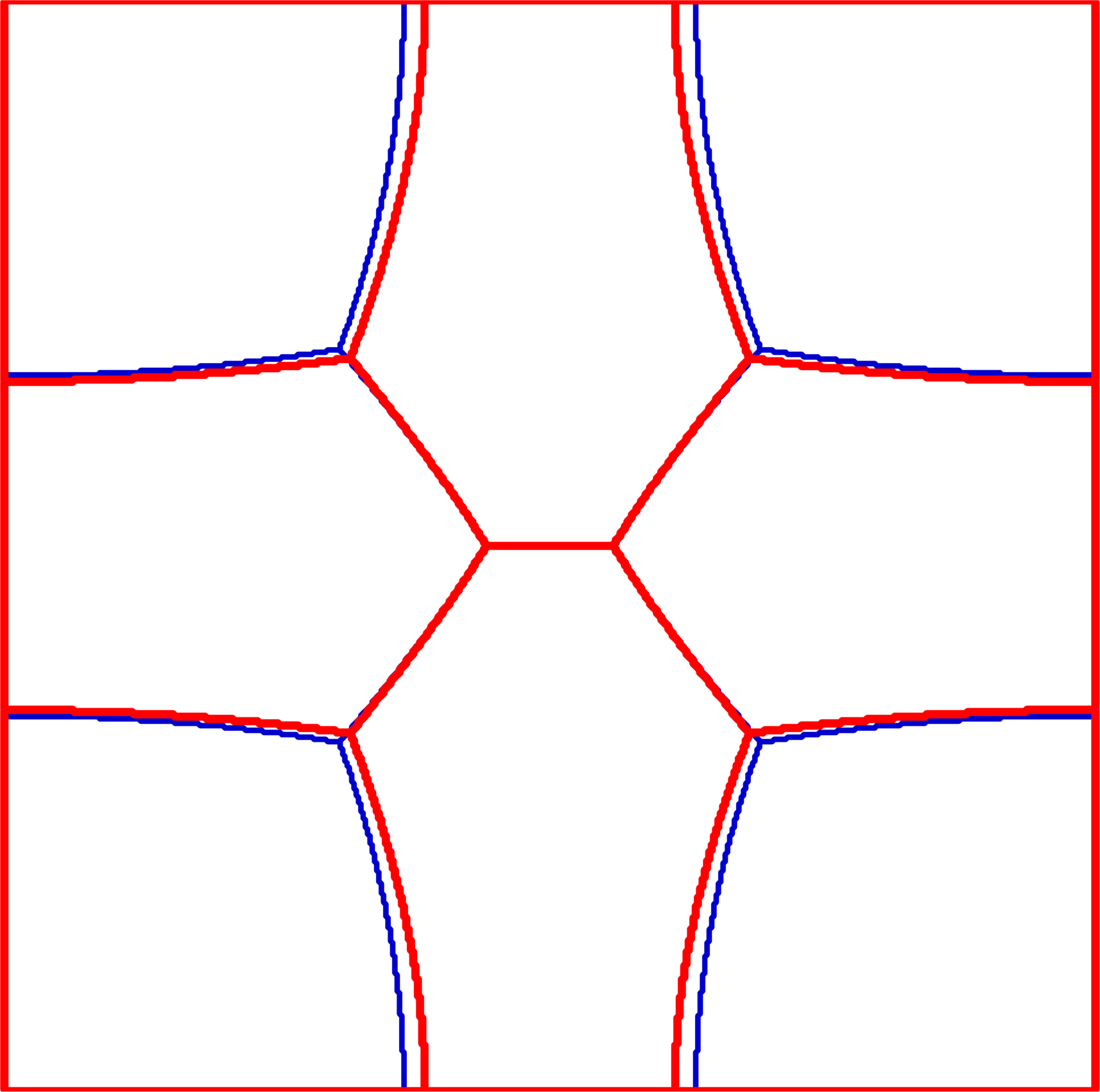}
& \includegraphics[width=0.19\textwidth]{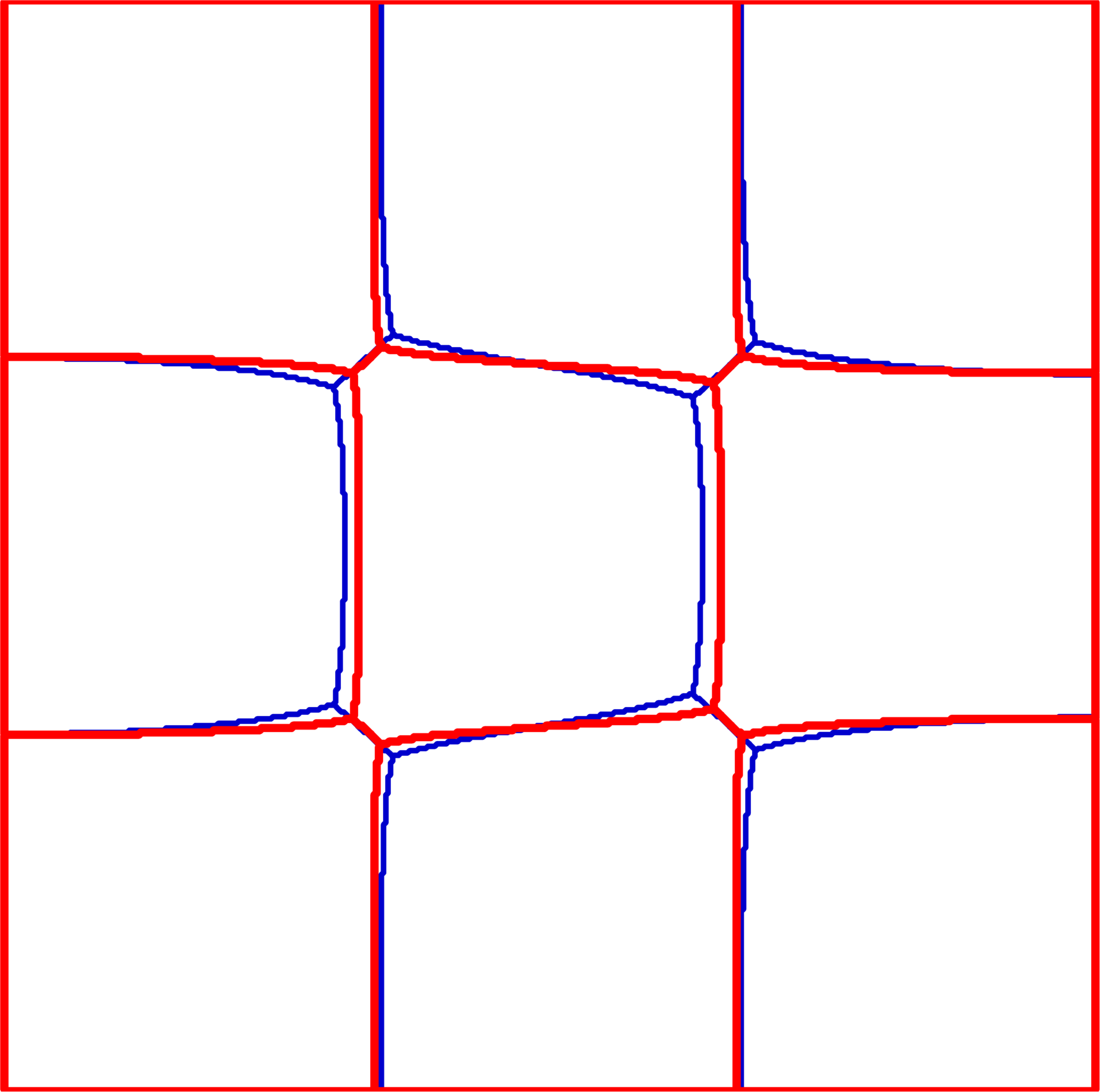}
& \includegraphics[width=0.19\textwidth]{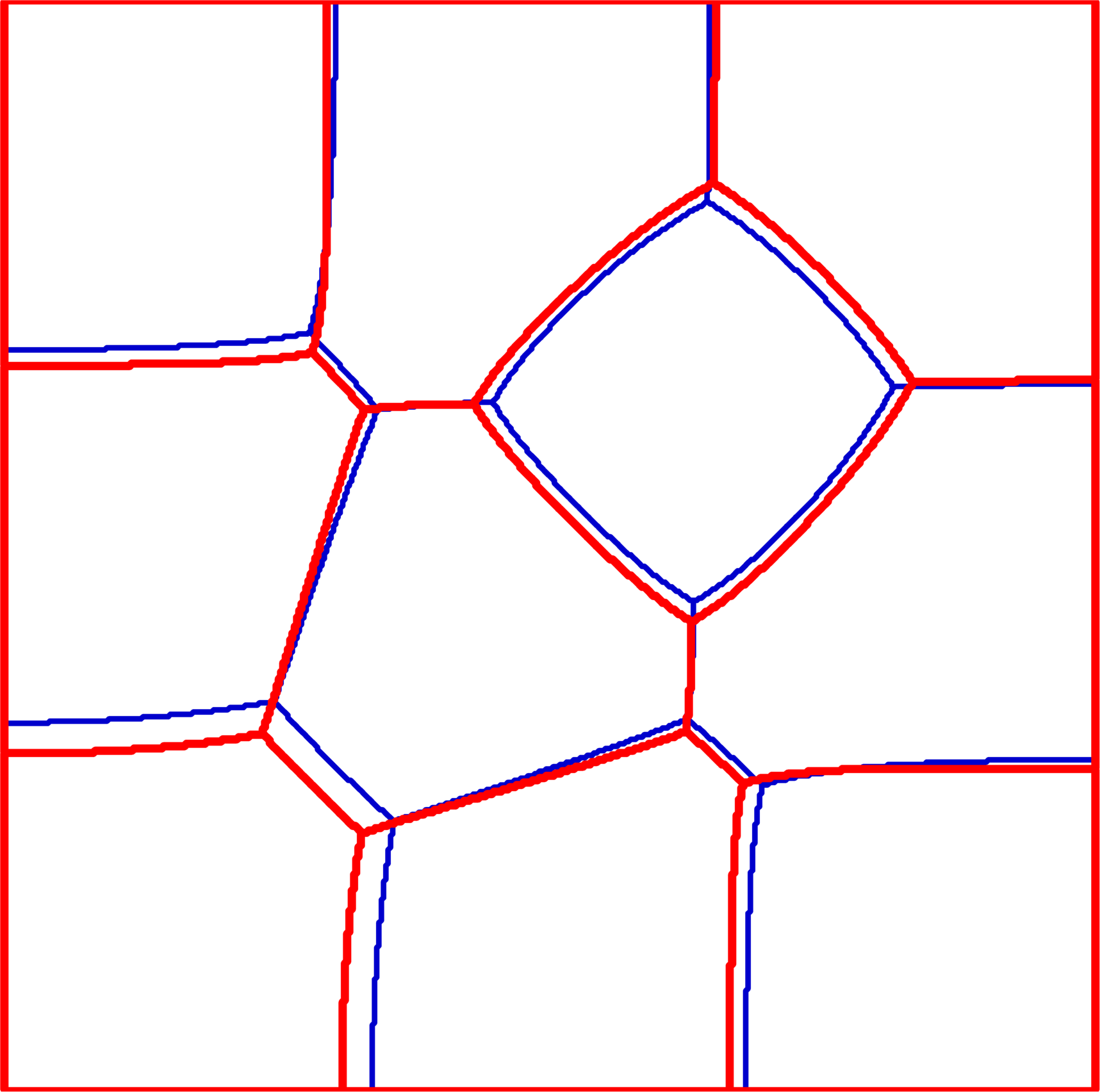}\\
$k=6$ & $k=7$ & $k=8$ & $k=9$ & $k=10$
\end{tabular}
\caption{Candidates for the max (red) and the sum (blue).\label{fig.carresummax}}
\end{figure}

In these figures, we observe that the candidates are very close except for two cases: 
\begin{itemize}
\item the equilateral triangle and $k=4$ : in that case, the optimal partition for the max is known and composed of four equilateral triangles. The candidate for the sum has less symmetry and has 4 singular points on the boundary and two interior singular points. These points collapse two by two to give 3 singular points on the boundary for the max.
\item the square for $k=7$: in that case, the energies of two candidates are very close and it is very difficult to conclude. These two candidates already appear in \cite{CyBaHo}.
\end{itemize}
In Table~\ref{synthese-sum}, we give the energies of the optimal partition for the sum obtained by the iterative method. We can compare these values with those of Table \ref{tab.Linfty} for the max. This is completely coherent with the previous section: when the criteria shows that the candidate to be optimal for the max can not be optimal for the sum, we obtained new partition with lower energy and which is not an equipartition.
\begin{table}[h!]
\begin{center}
\begin{tabular}{|c|c|c|c|c|c|c|c|c|c|}
\hline
$k$ & 2 & 3 & 4 & 5 & 6 & 7 & 8 & 9 & 10\\
\hline
$\triangle$ & 121.65 & 143.38 & 207.75 & 251.69 & 277.06 & 338.01 & 387.72 & 426.25 & 453.87 \\
\hline
$\ocircle$ & 14.68 & 20.19 & 26.37 & 33.21 & 38.95 & 44.02 & 50.44 & 57.69 & 63.67 \\
\hline
$\square$ & 49.40 & 66.14 & 78.956  &  103.37 & 125.68 & 144.34  & 159.93 & 177.68 & 201.85 \\
\hline
\end{tabular}
\end{center}
\caption{Lowest energies for the sum with the iterative algorithm.}
\label{synthese-sum}
\end{table}

\subsection{Partitions with curved polygons}
In this section we use, when available, an explicit representation for the partitions in order to exhibit candidates for the sum which are better than those for the max. We notice that for $k\in\{6,8,9\}$ for the disk and $k\in\{3,5\}$ for the square the partition has a simple symmetric structure. It is possible to approximate each of these partitions by assuming that the boundaries are either segments or arcs of circles. One important aspect used in the construction is the equal angle property which implies that all angles around a singular point are equal.

In the case of the disk, for $k\in \{6,8,9\}$ we obtained partitions which have a rotational symmetry of angle $2\pi/(k-1)$. Furthermore, we have a central domain which is a regular curved polygon with $k-1$ sides and $k-1$ external domains, each obtained by joining the edges of the polygons to the boundary of the disk with straight segments. This suggested us to see what happens if we assume that the rounded polygon's edges are arcs of circle. We note that this is not known theoretically. Furthermore, we want that pairs of consecutive arcs make an angle of $2\pi/3$. The general configurations are represented in Figure \ref{curved_polyk6} for $k=6$ and Figure \ref{curved_polyk8} for $k=7,8$.
We describe now the procedure of constructing such a rounded polygon for $k=6$. Let's start with an angle $\angle BAB'$ of measure $2\pi/5$ with $AB'=AB=\ell$. Note that $\ell \in (0,1)$ will be one of the parameters of the problem since we wish to vary these partitions in function of the size of the inner curved polygon. We wish to construct an arc $\arc{BB'}$ which makes equal angles of $\pi/3$ with $AB$ and $AB'$ (this is so that after symmetrization we have a pentagon with $2\pi/3$ angles. In order to construct this arc $\arc{BB'}$ we need the position of the center $C$ of the corresponding circle and the radius of this circle $CB=CB'=R$. Note that $\angle(CB,\arc{BB'}) = \pi/2$ and $\angle(AB,\arc{BB'})=\pi/3$, which implies that $\angle ABC=\pi/6$. By symmetry, we have $\angle AB'C=\pi/6$. Once we know $\theta = \angle BAB'$ and $\ell$, it is possible to find all other elements of the triangle $ABC$ by using, for example, the sine theorem
\[ \frac{AB}{\sin \hat C} = \frac{BC}{\sin\hat A}=\frac{CA}{\sin \hat B} \Leftrightarrow \frac{\ell}{\sin(\frac{\theta}{2}-\frac{\pi}{6})}=\frac{R}{\sin\frac\theta2}=\frac{AC}{\sin\frac \pi6}. \]
Once $AC$ is known we can determine the position of the center $C$ and then we can trace an arc of a circle of radius $R$ going from $B$ to $B'$. Note that the measure of $\angle BCB'$ is also needed in the implementation and is equal to $\theta-\pi/3$. The situation is similar for $k\in \{8,9\}$ with the difference that now $C$ is on the other side of $BB'$ so that the curved polygon has inward arcs. The case $k=8$ is depicted in Figure \ref{curved_polyk8} and the arguments for finding $R$ and $AC$ are essentially the same. 
\begin{figure}[!h]
\centering
\includegraphics[width=0.9\textwidth]{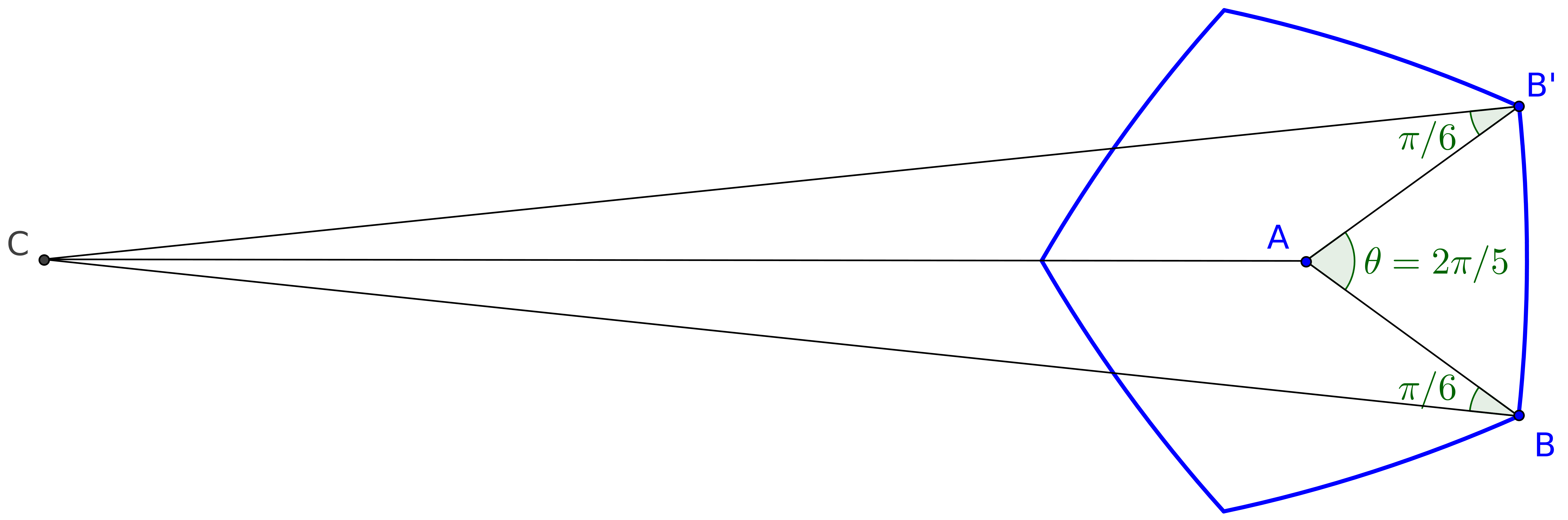}\\
\caption{Explicit construction of curved regular pentagon with angles equal to $2\pi/3$.}
\label{curved_polyk6}
\end{figure}
\begin{figure}[!h]
\centering
\includegraphics[width=0.9\textwidth]{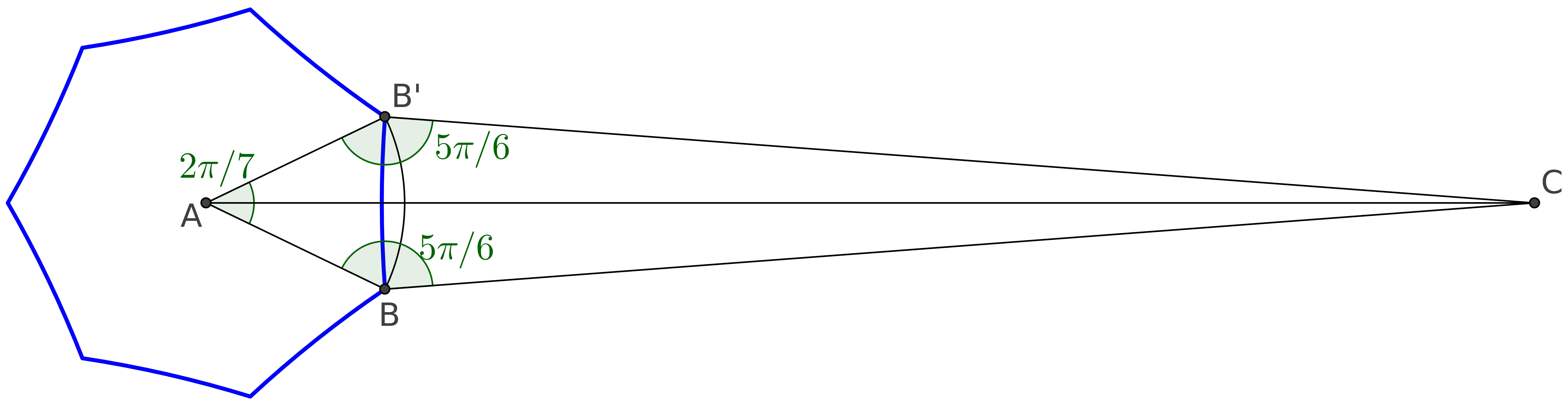}
\caption{Explicit construction of curved regular heptagons with angles equal to $2\pi/3$.}
\label{curved_polyk8}
\end{figure}
Note that once the arc $\arc{BB'}$ has been constructed, the exterior domain is drawn by continuing segments $AB,AB'$ until they reach the unit circle. Thus we have a way of constructing admissible partitions with the equal angle property which are very similar to the ones obtained with the iterative method. 
Once we fix $k \in \{6,8,9\}$ we can optimize the sum of the eigenvalues of the partition with respect to $\ell$. 
Let us remark that if we apply this method for $k=7$, we have necessarily a hexagon with straight lines as observed in Figure \ref{fig.disksummax} and analyzed in \cite{BBN16}.

The same method can be applied in the case of the square for $k=3,5$ (see Figure~\ref{curved_polys}). For $k=3$, the equal angle property applied at a boundary singular point imposes that the center $C$ of the circle is along a side of the square. 
As can be seen in Figure \ref{curved_polys} we denote by $B$  the triple point and $AB$ is the segment along the boundary of the partition cells along the symmetry axis. With these considerations, the equal angle property implies that $\angle ABC = \pi/6$. See Figure \ref{curved_polys} for more details. 
For $k=5$, we note that we have $4$ axes of symmetry and the central domain resembles a curved polygon with $4$ sides. We apply the previous arguments (see Figures \ref{curved_polyk6}-\ref{curved_polyk8}) to construct a regular polygon with $4$ sides and angles of measure $2\pi/3$. Then we extend this polygon to a partition of the square like in Figure \ref{curved_polys}. In both cases, we optimize the sum of the eigenvalues of these partitions with respect to the length $\ell=AB \in (0,1)$. 
\begin{figure}[h!t]
\centering 
\includegraphics[height=0.3\textwidth]{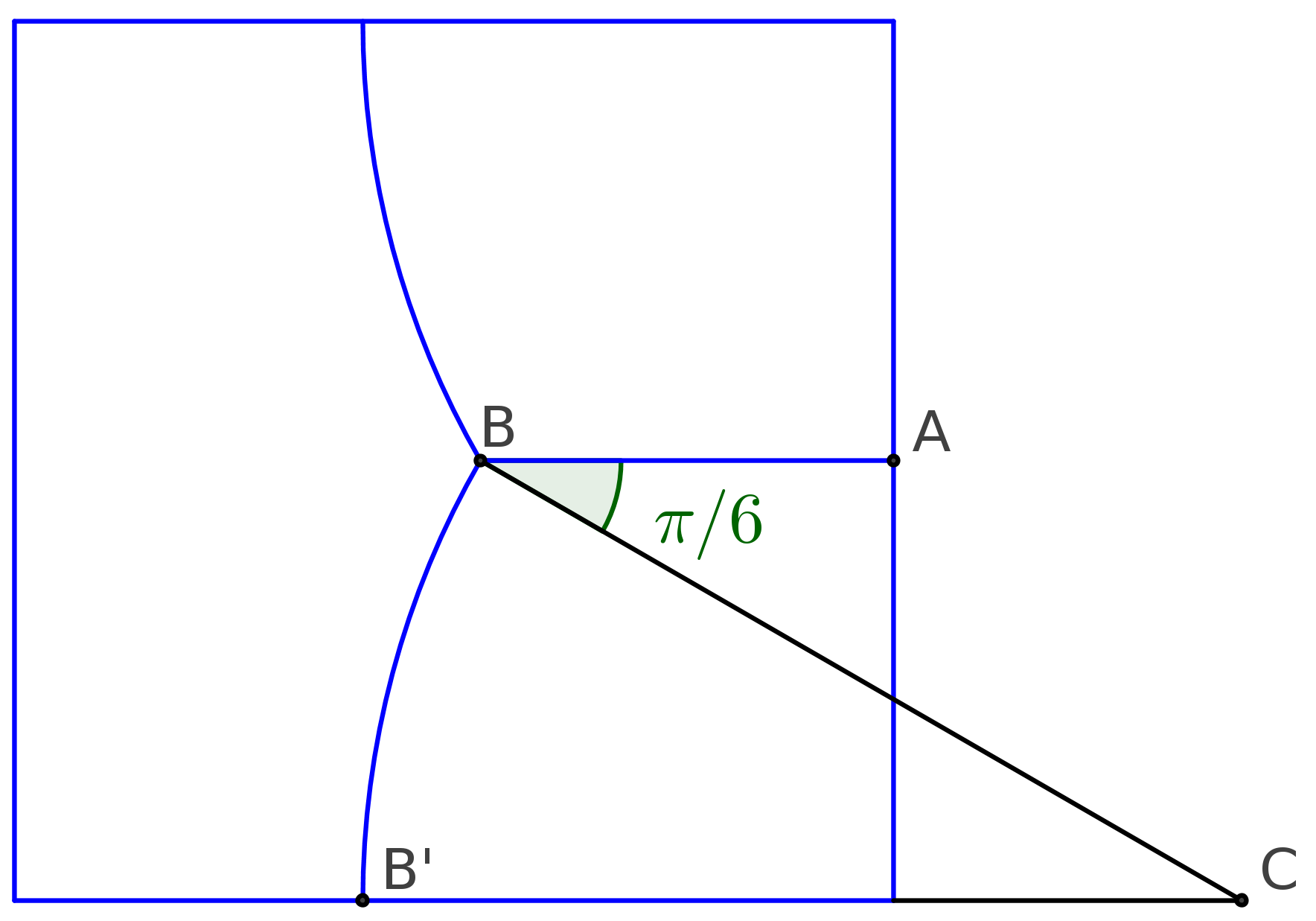}
\hspace{1cm}\includegraphics[height=0.3\textwidth]{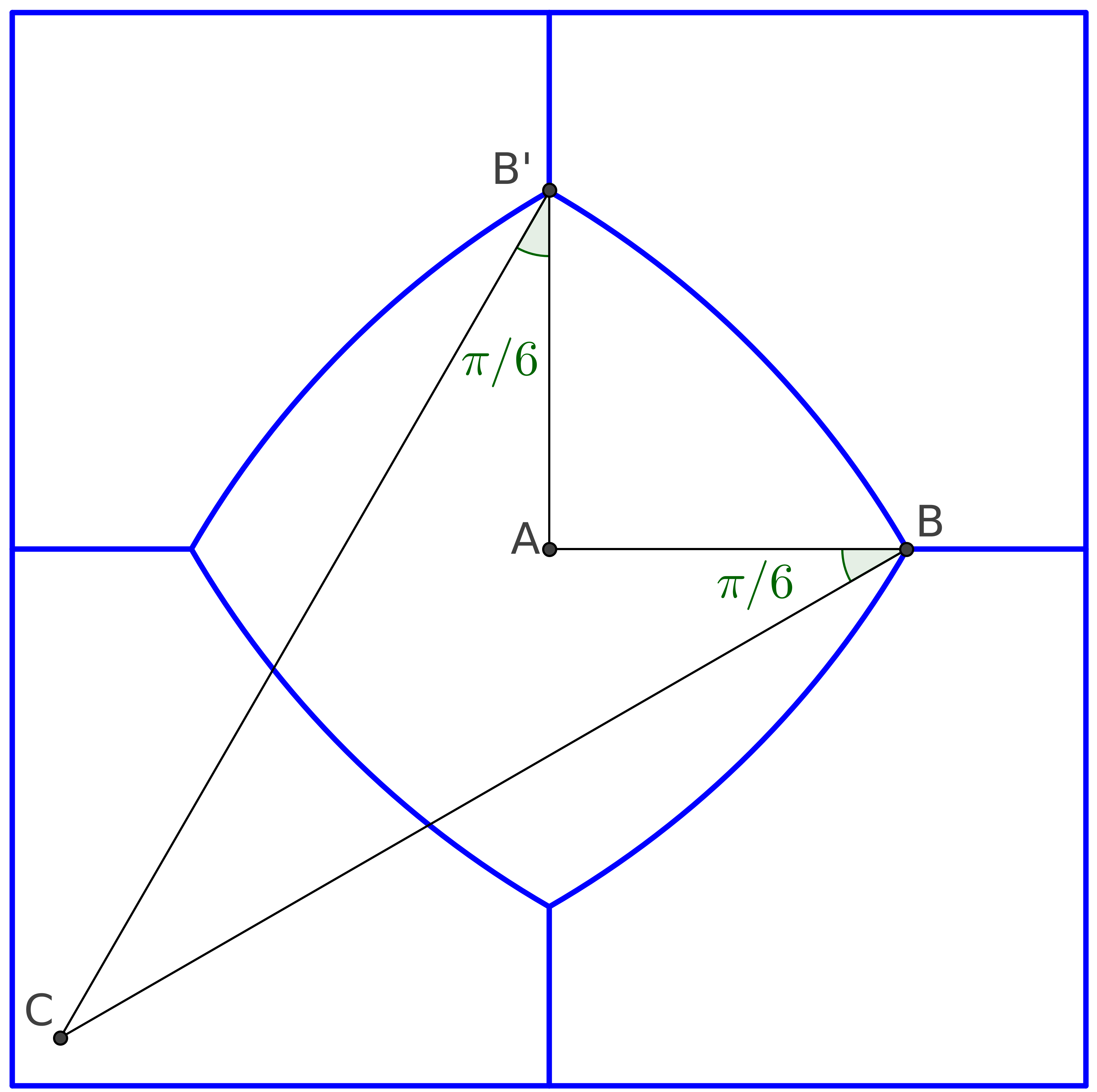}
\caption{Explicit construction of a $3$- and $5$-partition of the square with curved polygons.\label{curved_polys}}
\end{figure}

We implemented this method in FreeFem++  for the five configurations and Table~\ref{tab.encurv} presents the results.
We note that in every cases except for the square and $k=3$, the average of the eigenvalues is smaller than the energy obtained using the iterative method (see Table~\ref{synthese-sum}). However, the partitions we obtained with the two methods are close, as can be seen in Figure \ref{expl-vs-iter}.
\begin{table}[h!]
\begin{center}\begin{tabular}{|c|c|c| c | c| c|}
\hline
& \multicolumn{3}{c|}{Disk} &\multicolumn{2}{c|}{Square} \\
\hline
$k$ & 6 & 8 & 9 & 3 & 5\\
\hline
$\ell$ &0.411 & 0.3975 & 0.3981 & 0.4781 &0.5093\\
\hline
$\frac 1k\sum_{k}\lambda_{1}(D_{i})$ & 38.85 & 50.29 & 57.51 & 66.16 & 103.34\\
\hline
\end{tabular}\end{center}
\caption{Energies of the partition with constant curvatured boundaries.\label{tab.encurv}}
\end{table}

\begin{figure}[h!t]
\centering
\includegraphics[width=0.19\textwidth]{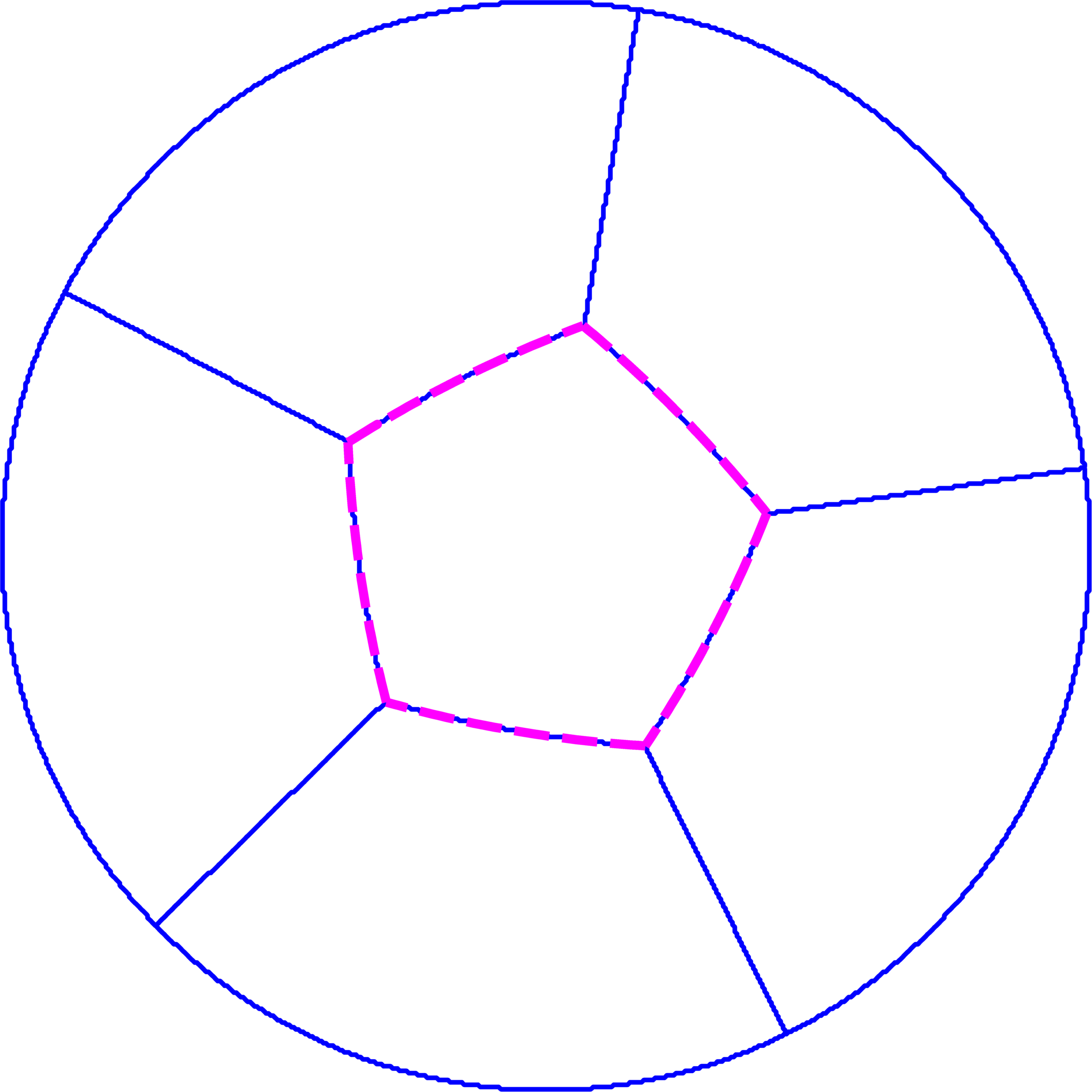}
\includegraphics[width=0.19\textwidth]{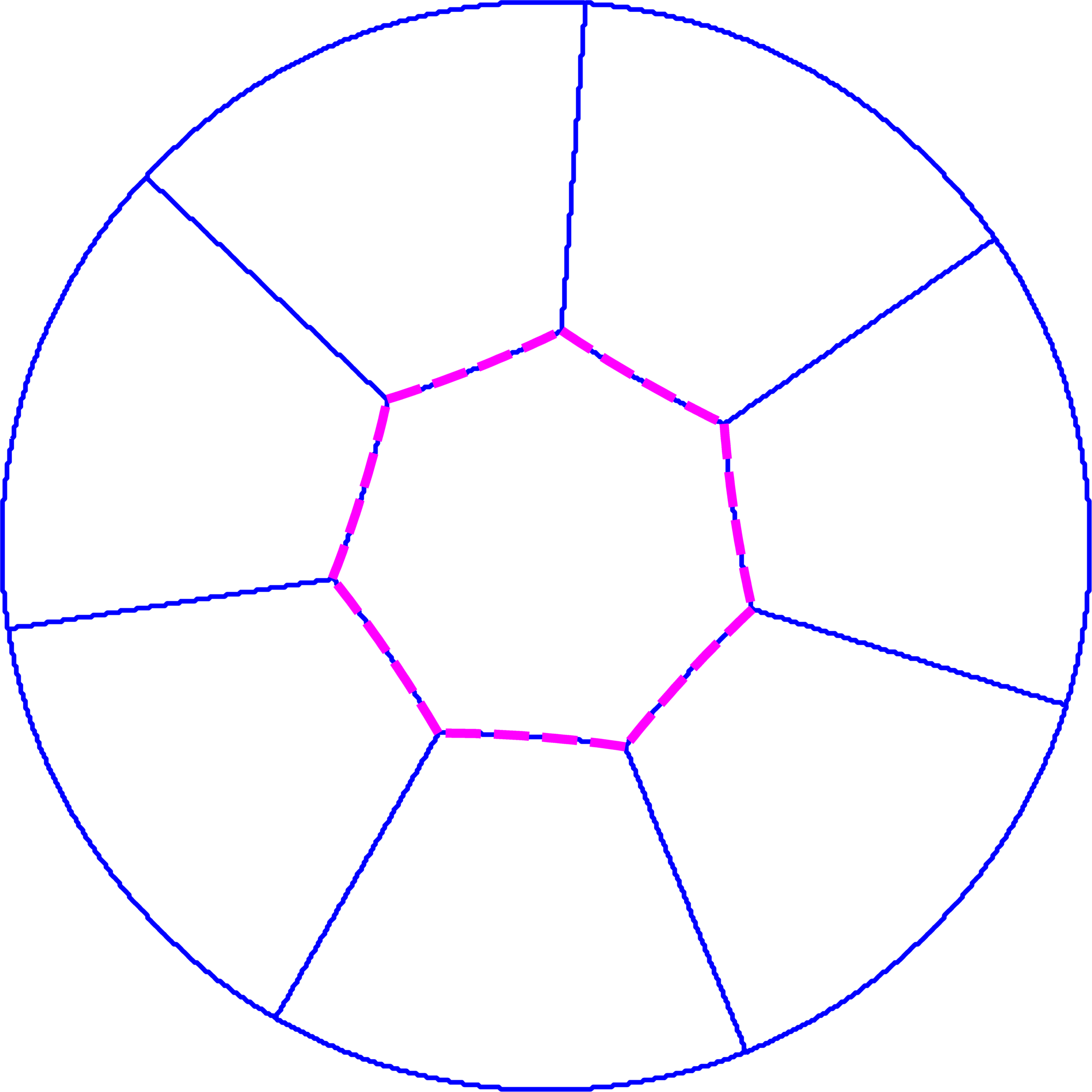}
\includegraphics[width=0.19\textwidth]{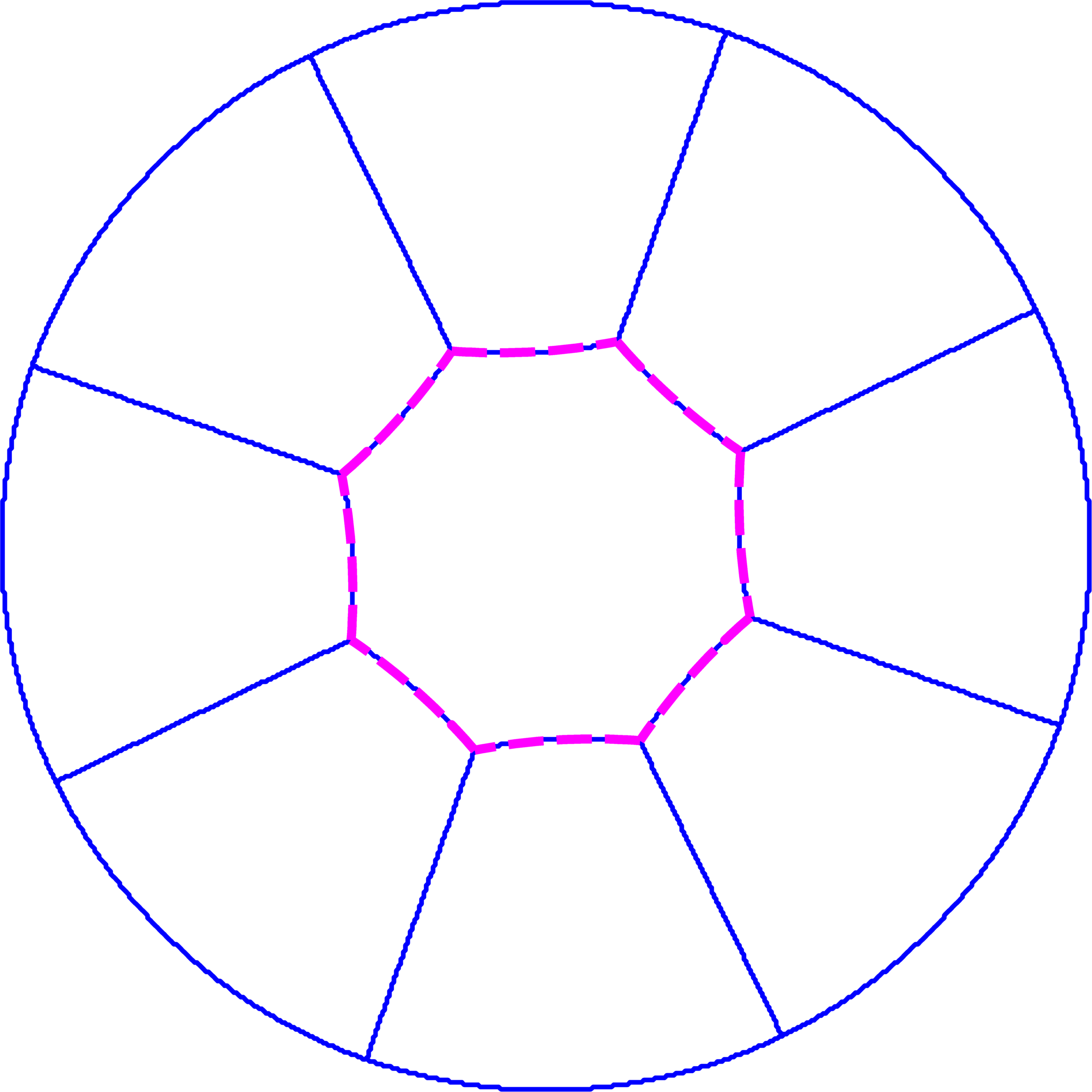}
\includegraphics[width=0.19\textwidth,angle=270,origin=c]{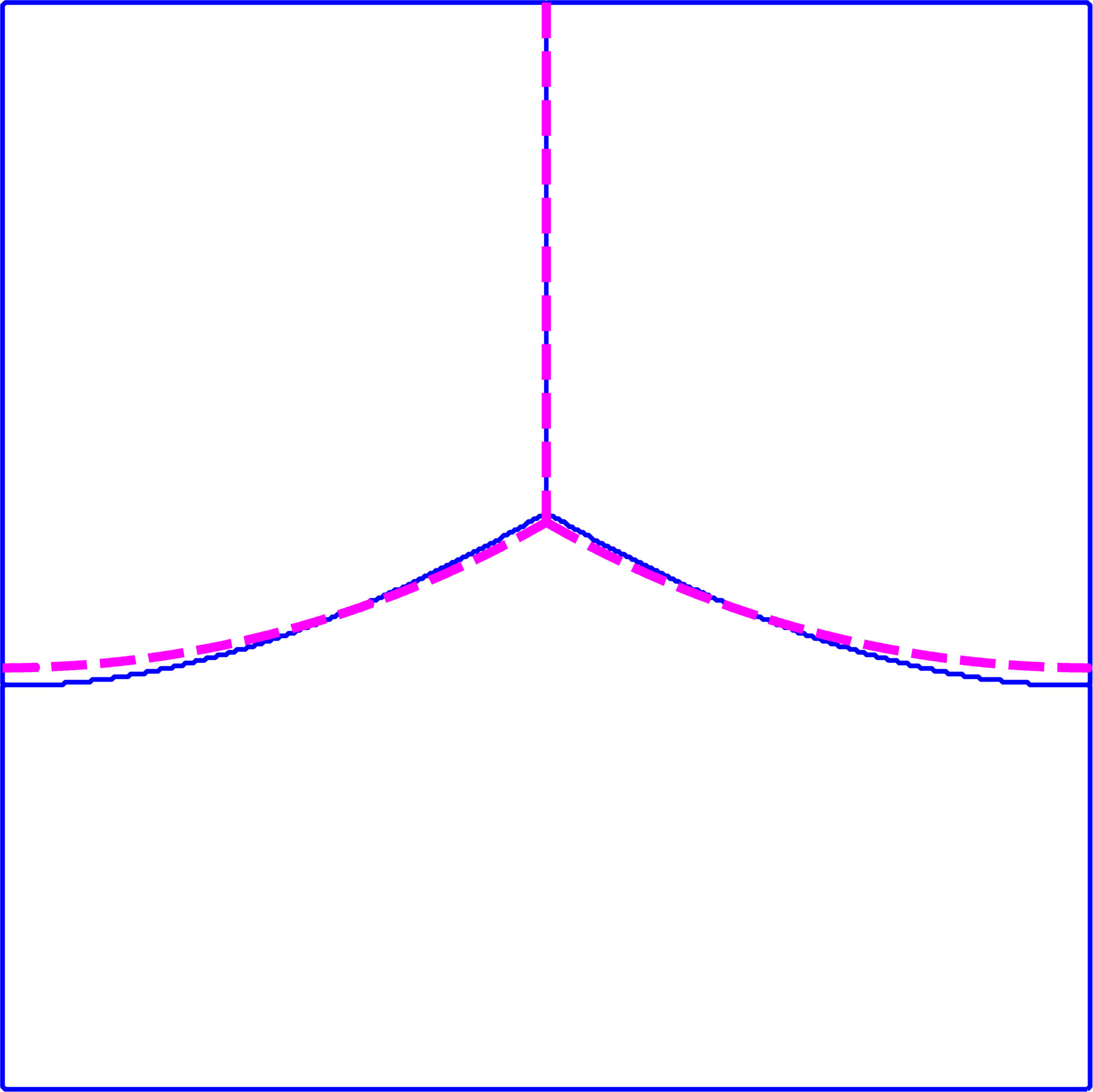}
\includegraphics[width=0.19\textwidth]{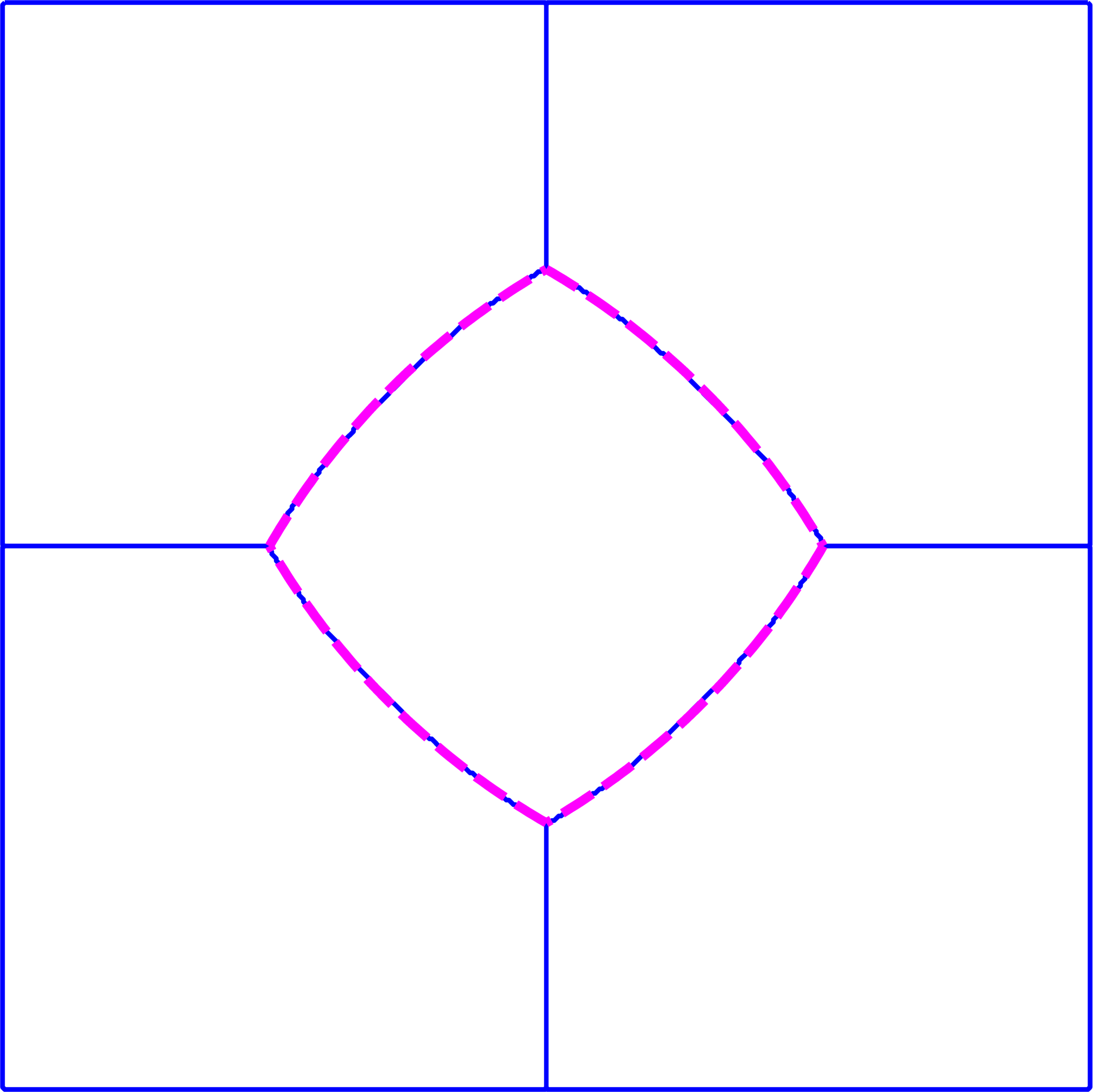}
\caption{Explicit partitions (magenta) vs partitions with iterative algorithm (blue).}
\label{expl-vs-iter}
\end{figure}


\ack This work is supported by a public grant overseen by the French National Research Agency (ANR) as part of the ``Investissements d'Avenir'' program (LabEx {\it Sciences Math\'ematiques de Paris} ANR-10-LABX-0098 and ANR {\it Optiform} ANR-12-BS01-0007-02).


%
  
\bibliographystyle{acmurl}



\begin{address}
  Beniamin Bogosel and Virginie Bonnaillie-No\"el \\
  D\'epartement de Math\'ematiques et Applications (DMA - UMR 8553) \\ 
  PSL Research University, ENS Paris, CNRS\\ 
  45 rue d'Ulm, F-75230 Paris cedex 05, France \\
  \texttt{beniamin.bogosel@ens.fr} and \texttt{bonnaillie@math.cnrs.fr}
\end{address}

\end{document}